\pdfoutput=1
\documentclass[12pt]{article}

 \usepackage{amsmath,amssymb,amscd,amsthm,esint}

\usepackage{graphics,amsmath,amssymb,amsthm,mathrsfs,amsfonts}
\usepackage{sidecap}
\usepackage{float}
\usepackage{extarrows}
\usepackage{booktabs}
\usepackage{verbatim}
\usepackage{hyperref}
\usepackage[usenames,dvipsnames]{xcolor}

\usepackage{titlesec}
\usepackage{titletoc}

\usepackage[titletoc]{appendix}
\usepackage{enumerate}

\newtheorem{thm}{Theorem}[section]
\newtheorem{lemma}[thm]{Lemma}
\newtheorem{cor}[thm]{Corollary}
\newtheorem{prop}[thm]{Proposition}

\newtheorem{remark}[thm]{Remark}

\newcommand{\varep}{\varepsilon}

\def\XXint#1#2#3{{\setbox0=\hbox{$#1{#2#3}{\int}$}
         \vcenter{\hbox{$#2#3$}}\kern-.5\wd0}}

\def\R{\mathbb{R}}

\def\e{\varepsilon}

\numberwithin{equation}{section}

\begin{document}

\title{On the Schr\"odinger equations with $B_\infty$ potentials in the region above a Lipschitz graph}

\author{Jun Geng\thanks{This work was supported in part by the NNSF of China (No. 12371096).}\footnote{
School of Mathematics and Statistics, Lanzhou University, Lanzhou, P.R. China. \newline
\indent\quad  E-mail: gengjun@lzu.edu.cn}
\and Ziyi Xu
\footnote{School of Mathematics and Statistics,
Lanzhou University,
Lanzhou, P.R. China.\newline
\indent \quad E-mail: 120220907780@lzu.edu.cn}}

\date{}
\maketitle

\begin{abstract}
  In this paper we investigate the $L^p$ regularity, $L^p$ Neumann and $W^{1,p}$ problems for generalized Schr\"odinger operator $-\text{div}(A\nabla )+ V $ in the region above a Lipschitz graph under the assumption that $A$ is elliptic, symmetric and  $x_d-$independent. Specifically, we prove that the $L^p$ regularity problem is uniquely solvable for $$1<p<2+\e.$$
  Moreover, we also establish the $W^{1,p}$ estimate for Neumann problem for
  $$\frac{3}{2}-\e<p<3+\e.$$
  As a by-product, we also obtain that the $L^p$ Neumann problem is uniquely solvable for $1<p<2+\e.$
  The only previously known estimates of this type pertain to the classical Schr\"odinger equation $-\Delta u+ Vu=0$ in $\Omega$ and $\frac{\partial u}{\partial n}=g$ on $\partial\Omega$ which was obtained by Shen \protect\cite{Shen-1994-IUMJ} for  ranges $1<p\leq 2$.
  All the ranges of $p$ are sharp.
\end{abstract}



\bigskip

\tableofcontents


\section{\bf Introduction}\label{section-1}
The regularity of solutions to the Schr\"odinger equations in the presence
of a positive potential is an important property underpinning the foundation of
quantum physics.
In this paper we will establish well-posedness of the $L^p$ regularity, $L^p$ Neumann and $W^{1,p}$ problems for generalized Schr\"odinger operator
  $$\mathcal{L}+V=-\text{div}(A\nabla )+ V$$
in the region above a Lipschitz graph.  Precisely, let $$\Omega=\{(x',t): x'\in \mathbb{R}^{d-1}, t\in\mathbb{R} \text{ and }t>\Psi(x')\}\subset\mathbb{R}^{d},\quad d\geq3$$
where $\Psi:\mathbb{R}^{d-1}\to\mathbb{R}$ is a Lipschitz function with $\|\nabla\Psi\|_{L^\infty(\mathbb{R}^{d-1})}\leq M$.

Following the tradition and
physical significance, we will refer to $V$ as the
electric potential. We assume that the coefficient matrix $A(x)=(a_{ij}(x))$ is real, and satisfies the ellipticity condition,
\begin{equation}\label{ellipticity}
\mu|\xi|^2\leq a_{ij}\xi_i \xi_j,\quad \|A\|_{L^\infty}\leq \mu^{-1}\quad {\rm for~~ any}\quad
\xi=(\xi_i)\in\mathbb{R}^{ d} ,
\end{equation}
where $\mu>0$, the symmetric condition, i.e.,
\begin{equation}\label{symmetric}
A=A^*,
\end{equation}
and the structure condition, i.e., $A$ is $x_d$-independent,
\begin{equation}\label{independent}
  A(x',s)=A(x',t)=A(x^\prime)\quad\text{for all }x'\in\mathbb{R}^{d-1}\text{ and }s,t\in\mathbb{R}.
\end{equation}
Moreover, we impose smoothness condition on the coefficient matrix, i.e., $A$ is partially Dini continuous with respect to $x'$. Indeed, assume that
\begin{equation}\label{Dini}
\int_0^t\frac{\eta(\rho)}{\rho}\,d\rho<\infty, \quad\text{for any }t>0.
\end{equation}
where
$$\eta(\rho)=\sup_{x_0\in\R^d, r\leq\rho} \left(|B(x_0,r)|^{-2}\int_{B(x_0,r)\cap\Omega} \int_{B(x_0,r)\cap\Omega} |A(x',x_d)-A(y',x_d)|^2\,dxdy\right)^{\frac12}$$
is the modulus of continuity.

Throughout this paper, we assume that $0<V\in B_\infty$, i.e., $V\in L_{loc}^\infty(\mathbb{R}^d)$, and there exists a constant $C>0$ such that, for all ball $B\subset \mathbb{R}^d$
\begin{equation}\label{V}
    \|V\|_{L^\infty(B)}\leq  C\fint_B  V.
\end{equation}
Examples of $B_\infty$ weights are $|x|^a$ with $0\leq a<\infty$.

In this paper we shall be interested in sharp regularity estimates of $u$, assuming that the data is in $L^p$. To state main results of the paper, denote $n$ the outward unit normal to $\partial\Omega$ and $\frac{\partial u}{\partial \nu}=n\cdot A\nabla u$ conormal derivative of $u$ associated with $\mathcal{L}$.

\begin{thm}\label{main-thm-W1p}($W^{1,p}$ estimate)
Let $\Omega$ be the region above a Lipschitz graph. Suppose $A$ satisfies \protect\eqref{ellipticity}-\protect\eqref{Dini} and $V(x)>0$ a.e. satisfies \protect\eqref{V}. Let $$\frac{3}{2}-\e<p<3+\e$$
and $u\in W^{1,p}(\Omega)$ be the weak solution to
\begin{equation}\label{W1p}
\mathcal{L} u+V u ={\rm div} f\quad  \text{in } \Omega,\quad\text{ and }\quad\frac{\partial u}{\partial \nu}  =-f\cdot n+G \quad\text{on } \partial\Omega,
\end{equation}
where $f\in L^{p}(\Omega,\mathbb{R}^d)$ and $G\in B^{-\frac 1p,p}(\partial\Omega)$. Then
$$\|\nabla u\|_{L^p(\Omega)}+\|V^\frac{1}{2}u\|_{L^p(\Omega)}\leq C\left\{\|f\|_{L^p(\Omega)}+\|G\|_{B^{-\frac 1p,p}(\partial\Omega)}\right\}$$
where $C$ depends only on $\mu$, $d$, $p$ and the Lipschitz character of $\Omega$.
\end{thm}

\begin{remark}
Let $\Omega,A$ and $V$ be the same as in Theorem \protect\ref{main-thm-W1p}. It is not hard to see the $W^{1,p}$ estimate also holds for Schr\"odinger equation subjected to Dirichlet boundary condition. Precisely, let $$\frac{3}{2}-\e<p<3+\e$$
and $u\in W^{1,p}(\Omega)$ be the weak solution to
\begin{equation}
\mathcal{L} u+V u ={\rm div} f\quad  \text{in } \Omega,\quad\text{ and }\quad u  =G \quad\text{on } \partial\Omega,
\end{equation}
where $f\in L^{p}(\Omega,\mathbb{R}^d)$ and $G\in B^{1-\frac 1p,p}(\partial\Omega)$. Then
$$\|\nabla u\|_{L^p(\Omega)}\leq C\left\{\|f\|_{L^p(\Omega)}+\|G\|_{B^{1-\frac 1p,p}(\partial\Omega)}\right\}$$
where $C$ depends only on $\mu$, $d$, $p$ and the Lipschitz character of $\Omega$. The proof follows the similar line of argument as in Theorem \protect\ref{main-thm-W1p}, and we omit the details.
\end{remark}

In next theorem, the nontangential maximal function of $u$ at $z$ is denoted as
$$(u)^*(z)=\sup\{|u(x)|: x\in \Omega\text{ and } |x-z|<C\text{dist}(x,\partial\Omega)\},\quad z\in\partial\Omega.$$
Moreover, the maximal function $m(x, V)$ denotes the Fefferman-Phong-Shen maximal function and we refer to Section 2 for the exact definition.

\begin{thm}\label{main-thm-Lp-r}(Nontangential maximal function estimate)
Let $\Omega$ be the region above a Lipschitz graph. Suppose $A$ satisfies \protect\eqref{ellipticity}-\protect\eqref{Dini} and $V(x)>0$ a.e. satisfies \protect\eqref{V}. Let $$1< p<2+\e.$$ Then for every $g$ satisfying
$$\int_{\partial\Omega}|\nabla_{tan}g|^p d\sigma +\int_{\partial\Omega}|gm(x,V)|^p d\sigma<\infty,$$
there exists a unique solution $u$ so that
\begin{equation}\label{RP}
\mathcal{L} u+V u =0 \quad  \text{in } \Omega,\quad
u  =g\quad\text{on } \partial\Omega\quad\text{and}\quad
(\nabla u)^*\in L^p(\partial\Omega).
\end{equation}
 Moreover, the solution $u$ satisfies
\begin{equation}\label{rp-estimate}
\|(\nabla u)^*\|_{L^p(\partial\Omega)}\leq C\|\nabla_{tan}g\|_{L^p(\partial\Omega)} +C\|gm(x,V)\|_{L^p(\partial\Omega)},
\end{equation}
where $C$ depends only on $\mu$, $d$, $p$ and the Lipschitz character of $\Omega$.
\end{thm}

For Laplace's equation $\Delta u=0$ in bounded Lipschitz domains, the $L^p$ Dirichlet, Neumann and regularity boundary value problems are well understood. Indeed, it has been known since the early 1980s that the Dirichlet problem with optimal estimate $(u)^*\in L^p$ is uniquely solvable for $2-\e<p<\infty$ and the Neumann problem as well as the regularity problem with optimal estimate $(\nabla u)^*\in L^p$
is uniquely solvable for any $1<p<2+\e$ where $\varep>0$ depends on
$\Omega$. If $\Omega$ is a $C^1$ domain, $L^p$ Dirichlet and regularity problems
are solvable for any $1<p<\infty$ \protect\cite{FJR}. And the ranges of
$p$ are sharp (see \protect\cite{Dahlberg-Kenig-1987,JK-1981, Verchota-1984};
also see \protect\cite{GS-JFA-2010, Kenig-book,Kim-Shen-JFA-2008} for references on related work on boundary value problems in bounded Lipschitz domains or bounded convex domains).

However, for a general second order elliptic operator
$\mathcal{L}$,
$L^p$ Dirichlet problem (thus $L^{p^\prime}$ regularity problem)
may not be solvable for any $1<p<\infty$,
even if the coefficients of $\mathcal{L}$ are continuous and
$\Omega$ is smooth.
Furthermore, it is known that the $L^{p}$ Dirichlet problem
 for $\mathcal{L}$ on $\Omega$
is solvable if and only if
the $\mathcal{L}$-harmonic measure
is a $B_p$ weight with respect to the surface measure
on $\partial\Omega$. We
refer the reader to \protect\cite{Kenig-book}
for references on these and other deep results on the solvability of
the $L^p$ Dirichlet problem. Concerning
the $L^p$ regularity problem for
general second order elliptic operators, we mention that the study
was initiated by Kenig and Pipher \protect\cite{KP-1993, KP-1995}.
\begin{remark}
It worths pointing that, regarding the
conditions on the $x_d$ variable, the global estimate $\|(u)^*\|_{L^p(\partial\Omega)}\leq C\|f\|_{L^p(\partial\Omega)}$ or $\|(\nabla u)^*\|_{L^p(\partial\Omega)}\leq C\|f\|_{L^p(\partial\Omega)}$ may fail for any $p$, even if $A(x^\prime, x_d)$ satisfies \protect\eqref{ellipticity} and \protect\eqref{symmetric} and is $C^\infty$
in $\mathbb{R}^{d+1}$.
\end{remark}
If the coefficient matrix
$A$ is $x_d$-independent, the $L^p$ solvability of the Neumann and regularity problems in the case of radially independent coefficients in the unit ball was essentially established in \protect\cite{KP-1993} for the sharp range $1< p< 2+\e$.
 Additionally, \protect\cite{KKPT-2000, KP-1995, KR-2008} contains related research on $L^p$ boundary value problems for operators with $x_d$-independent, yet non-symmetric coefficients. We also remark here in the scalar case, under the assumption that $A$ is periodic and Dini continuous in $x_d$ variable, Kenig and Shen \protect\cite{KS-2011-1} obtain the sharp $L^p$ Dirichlet, Neumann and regularity estimates. In the case of elliptic systems, $L^2$ Dirichlet, Neumann and regularity estimates were obtained in \protect\cite{KS-2011} by assuming that $A$ is periodic and H\"older continuous.

The $L^p$ solvability with
$1<p\leq 2$ for Schr\"odinger equations $-\Delta u+Vu=0$ in $\Omega$ subjected to Neumann boundary $\frac{\partial u}{\partial n}=f$ on $\partial\Omega$ in Lipschitz graph domains
was first obtained by Shen \protect\cite{Shen-1994-IUMJ} (related work about the Schr\"odinger operator $-\Delta +V$ may be found in \protect\cite{Auscher-2007,MT-1999-JFA,Shen-2024,Shen-1995,Shen-1999}). We extend Shen's result to the generalized Schr\"odinger operator $\mathcal{L}+ V $ and show that the $L^p$ regularity and Neumann problems are uniquely solvable for sharp ranges $1<p< 2+\e$. We also establish that the $W^{1,p}$ estimate holds for sharp ranges
$\frac{3}{2}-\e<p<3+\e.$

Our proof to Theorem \protect\ref{main-thm-Lp-r} was motivated by the
work in \protect\cite{Shen-1994-IUMJ}. We first seek to establish the global $L^2$ Rellich-Nec$\breve{a}$s-Payne-Weinberger
type estimate
$$
\int_{\partial\Omega}|\nabla u|^2d\sigma\sim C \int_{\partial\Omega}|\nabla_{tan} u|^2d\sigma+C\int_{\partial\Omega}|um(x,V)|^2d\sigma.
$$
Next, let $f\in L^p(\partial\Omega)$ and $\mathcal{S}(f)$ denote the single layer potential with density $f$, defined by
\begin{equation}\label{S}
\mathcal{S} (f)(x)=\int_{\partial\Omega}\Gamma (x,y) f(y)d\sigma(y)
\end{equation}
where $\Gamma(x,y)$ denotes the fundamental solution for operator $\mathcal{L}+V$ on $\mathbb{R}^d$. Let $u(x)=\mathcal{S}(f)(x)$, then $\mathcal{L} u +V u=0$ in $\mathbb{R}^d\backslash\partial\Omega$. With the help of the method of layer potential, we were able to show that
$\mathcal{S}: L^2(\partial\Omega)\to \widetilde{W}^{1,2}(\partial\Omega)$  is invertible where $\widetilde{W}^{1,2}(\partial\Omega)=\{f: \|\nabla_{tan}f\|_{L^2(\partial\Omega)}+ \|fm(x,V)\|_{L^2(\partial\Omega)}<\infty\}$. We also remark here that the boundary $L^\infty$ estimate
 $$
\sup_{x \in D(x,r)}|u(x)| \leq \frac{C_{k}}{\left(1+r m\left(x_{0}, V\right)\right)^{k}}\left( \fint_{D(x,2r)}|u(x)|^{2} d x\right)^{1 / 2}\quad\text{for } k\geq1,
$$
obtained by Moser's iteration method, is needed in the estimates for the Neumann function and the fundamental solution.  This leads to the $L^2$ estimate. Our proof of the $L^p(p\neq 2)$ estimates
with the range $1<p<2+\e$ is completed by two different methods. The result for $2<p<2+\e$ is done by a real variable argument, using the $L^2$ solvability and the dual Calder\'on-Zygmund decomposition as well as a variant of the `Good-$\lambda$' inequality. The extension to $1<p<2$ is accomplished by showing that the solution $u$, of the regularity problem with data $g\in\mathcal{H}_{1,at}^1(\partial\Omega)$, satisfies
$$
\int_{\partial\Omega}|(\nabla u)^*|\,d\sigma\leq C\|g\|_{\mathcal{H}_{1,at}^1(\partial\Omega)}
$$
where $\mathcal{H}_{1,at}^1(\partial\Omega)$ is the atomic Hardy space on $\partial\Omega$. This is finished by establishing estimate of
the non-tangential maximal function of $\nabla u$ with $L^2$ atoms data. Then the full result
follows by interpolation theory.

The main ingredients in our approach to Theorem \ref{main-thm-W1p} are interior estimate of $\mathcal{L} u +V u=0$ and the $L^2$ non-tangential maximal function estimate as well as the solvability of the regularity problem with data in $W^{1,q}(\partial\Omega)$ for some $q>2$. Let $v$ be weak solutions to $\mathcal{L}v+Vv=0$ with Neumann boundary conditions $\frac{\partial v}{\partial \nu}  =0$. Following the method in \cite{Geng-2012}(see also \cite{Shen-2007}), by using the real variable argument which was also used in the proof of $L^p$ regularity estimate for the case $2<p<2+\e$, we reduce the $W^{1,p}$ estimates of \eqref{W1p} to the weak reverse H\"older inequality
\begin{equation}
    \label{rh inequality1}
    \bigg(\fint_{D(x_{0},r)}|\nabla v|^p\,dx\bigg)^{\frac{1}{p}}\leq C\bigg(\fint_{D(x_{0},2r)}|\nabla v|^2dx\bigg)^{\frac{1}{2}}
\end{equation}
holds for $p>2$. To see \eqref{rh inequality1}, we re-write
$|\nabla v|^p  =|\nabla v|^2\cdot|\nabla v|^{p-2}$
and estimate $|\nabla v|^2$ by its $L^2$ non-tangential maximal function estimate. To handle the term $\int_{D(x_{0},2r)}|\nabla v|^{p-2}\,dx$, the key step is to establish the De Giorgi-Nash type estimate,
$$
|u(x)-u(y)| \leqslant C\left(\frac{|x-y|}{r}\right)^{\alpha} \sup _{B(x_{0}, r)}|u|
$$
for any $x, y \in B\left(x_{0}, r/ 4\right)$.
This is done in Section $3$, using a perturbation
scheme taken from \cite{Shen-1999}, with the help of size estimate of fundamental solution established in Section \ref{section-3}.
This gives \eqref{rh inequality1} for any $2<p< 3+\frac{\alpha}{1-\alpha}$. And the ranges of $\frac{3}{2}-\e<p<2$ follows by a duality argument.

This article is organized as follows. In Section \ref{section-2}, we collect some known results for the Fefferman-Phong-Shen maximal function. In Section \ref{section-3}, the boundary $L^\infty$ estimate and H\"older estimate are established, and with this, estimates for the fundamental solution and the Neumann function are established. Moreover, in Section \ref{section-4} we give the $L^2$ estimates of the  regularity problem by Rellich estimates. And $L^p$ regularity estimate and $L^p$ Neumann estimate are given in Section \ref{section-6} and Appendix respectively. Section \ref{section-7} presents $W^{1,p}$ estimate for the Neumann problem.

We end this section with some notations. We will use $A\sim B$ means that there exist constants $C>0$ and $c>0$ such that $c\leq A/ B\leq C$.
We use $\fint_E u$ to denote the average of $u$ over the set $E$; i.e.
$$
\fint_E u =\frac{1}{|E|} \int_E u.
$$
Let
\begin{equation}\label{D-D'}
D(x_0,r)=B(x_0,r)\cap\Omega\quad\text{and}\quad
\Delta(x_0,r)=B(x_0,r)\cap\partial\Omega,
\end{equation}
where $B(x_0,r)$ denotes the sphere centered at $x_0$ with radius $r$.



\section{\bf The Fefferman-Phong-Shen maximal function \texorpdfstring{$m(x,V)$}{}}\label{section-2}
In this section, we will quote a number of results from \cite{Shen-1994-IUMJ}.
Other versions of these lemmas and definitions appeared in \cite{Shen-1995}, and are related to the ideas of Fefferman and Phong \cite{Fefferman-1983}. Let $m(x,V)$ denote the Fefferman-Phong maximal function associated to $V$ satisfying \eqref{V}.
We first note that if $V$ satisfies \eqref{V}, then $V$ is a Muckenhoupt $A_\infty$ weight function \cite{Grafakos-book}. Recall that we say $\omega\in A_\infty$ if the $A_\infty$ characteristic constant of $\omega$
\begin{equation}\label{}
	[\omega]_{A_\infty}:=\sup_{B(x,r)\subset\mathbb{R}^d}\left\{\bigg(\fint_{B(x,r)}\omega\bigg) exp\bigg(\fint_{B(x,r)} \ln \omega^{-1} \bigg)\right\}<\infty.
\end{equation}
We list some properties of $A_\infty$ weight, as below.
\begin{lemma}\label{A-1}
	Let $\omega\in A_\infty$, then
	
	(1) the measure $\omega\,dx$ is doubling: precisely, $\omega$ satisfies
	\begin{equation}\label{doubling-condition}
		\omega(2B)\leq C\omega(B)
	\end{equation}
	for any ball $B$ in $\mathbb{R}^d$, where $\omega(B)=\int_B \omega$;
	
	(2) there exist $0<C_1,C_2<1$ such that for any ball $B$ in $\mathbb{R}^d$,
	\begin{equation}
		\left|\left\{x\in B:\omega(x)\geq C_1\fint_{B}\omega\right\}\right|\geq C_2|B|.
	\end{equation}

\end{lemma}
\begin{proof}
	See e.g. \cite{D-book}, \cite{Grafakos-book} or \cite{Stein-book}.
\end{proof}

For $V\in A_\infty$, the maximal\ function was introduced by Shen in \cite{Shen-1994-IUMJ}, motivated by the work of Fefferman and Phong in \cite{Fefferman-1983}.
For $x \in \mathbb{R}^d, r>0$, let
$$
\psi(x, r)=\frac{1}{r^{d-2}} \int_{B(x, r)} V(y) dy,
$$
and the Fefferman-Phong-Shen maximal function is defined as
\begin{equation}\label{maximal function}
m(x, V)=\inf _{r>0}\left\{\frac{1}{r}: \psi(x, r) \leq 1\right\}.
\end{equation}

Some auxiliary lemmas about properties of $m(x,V)$ referring the reader to  for a detailed presentation are list below. We omit the proofs in our exposition.

\begin{prop}\label{upper-bound-V}
If $V$ satisfies \eqref{V}, then for a.e. $x \in \mathbb{R}^{d}$ , 
 $0 < r \leq R  < \infty$,
\begin{equation*}
\psi(x, r) \leq C \bigg(\frac{r}{R}\bigg)^{2-\frac dq} \psi(x, R)
\end{equation*}
where $1<q\leq\infty$ and
\begin{equation*}
V(x) \leq C m(x, V)^{2}.
\end{equation*}
\end{prop}

\begin{proof}
  See \cite{Shen-1994-IUMJ}.
\end{proof}

\begin{prop}\label{prop2.6}
Assume that $V$ satisfies \eqref{V} and let $\hat{r}=\frac{1}{m(x,V)}$. Then
\begin{align*}
\psi\left(x, \hat{r}\right)=1.
\end{align*}
Moreover, $r\sim \hat{r}$ if and only if $\psi(x, r)\sim 1$.
\end{prop}
\begin{proof}
  See \cite{Shen-1994-IUMJ}.
\end{proof}

\begin{lemma}\label{lemma-function m}
There exist $C>0, c>0$ and $k_{0}>0$ depending only on $d$ and the constant in \eqref{V}, such that for $x,y$ in $\mathbb{R}^d$,
\begin{align}
m(x, V) &\sim m(y, V) \quad \text { if }|x-y| \leq \frac{C}{m(x, V)} \label{estimate-m1},\\
c(1+|x-y| m(x, V))^{-k_{0}} &\leq  \frac{m(x, V)}{m(y, V)}\leq C(1+|x-y| m(x, V))^{k_0 /\left(k_{0}+1\right)}\label{estimate-m2},\\
c(1+|x-y| m(y, V))^{1 /\left(k_{0}+1\right)} & \leq 1+|x-y| m(x, V)
 \leq C(1+|x-y| m(y, V))^{k_{0}+1} \label{estimate-m3}.
\end{align}
\end{lemma}

\begin{proof}
  See \cite{Shen-1994-IUMJ}.
\end{proof}

 We end this section with the Fefferman-Phong type estimate and we refer the reader to \cite{Shen-1994-IUMJ} or \cite{Shen-1995}.

\begin{lemma}\label{Fefferman-Phong}
  Let $u\in C_0^1(\R^d)$. Assume $V$ satisfies \eqref{V}. Then
    \begin{equation}\label{Fefferman-Phong-1}
      \int_{\Omega}|u(x)|^2m(x,V)^{2}\,dx\leq C\int_{\Omega}|\nabla u|^2\,dx+C\int_{\Omega}V|u|^2\,dx.
    \end{equation}
\end{lemma}
\begin{proof}
  See \cite[Lemma 1.11]{Shen-1994-IUMJ}. It follows from Proposition \ref{prop2.6} and Lemma \ref{lemma-function m} as well as a covering argument.
\end{proof}




\section{\bf De Giorgi-Nash type estimate and boundary \texorpdfstring{$L^\infty$}{} estimate}\label{section-3}

This section is devoted to a uniform H\"older estimate for the weak solutions of $(\mathcal{L}+V) u=0$.

\begin{thm}\label{Holer-estimate}
   Assume that $A$ satisfies \eqref{ellipticity}-\eqref{symmetric} and $V$ satisfies \eqref{V}. Suppose that $u \in W^{1,2}(B(x_{0}, 2R))$ is a weak solution of $\mathcal{L}u+V u=0$ in $B(x_{0}, 2R)$. Then
$$
|u(x)-u(y)| \leqslant C\left(\frac{|x-y|}{R}\right)^{\alpha} \sup _{B(x_{0}, 2R)}|u|
$$
for any $x, y \in B\left(x_{0}, R/2\right)$.
 \end{thm}

We begin with the boundary $L^\infty$ estimate.
\begin{lemma}\label{pointwise estimate}
 Suppose $A$ satisfies \eqref{ellipticity}-\eqref{symmetric} and $V(x)>0$ a.e. in $\mathbb{R}^d$.   Let $u$ be a weak solution of
	\begin{equation}\label{v_B}
		\mathcal{L} u+V u=0 \quad \text{in}\quad D(x_0,r) ~~\text{and} \\
		~~~~~\frac{\partial u}{\partial \nu} =0\quad \text{on}~~\Delta(x_0,r),
	\end{equation}
for some $x_0\in \overline{\Omega}$ and $r>0$. Then
$$
\sup_{x \in D(x_0,\frac r2)}|u(x)| \leq \frac{C_{k}}{\left(1+r m\left(x_{0}, V\right)\right)^{k}}\left( \fint_{D(x_0,r)}|u(x)|^{2} d x\right)^{1 / 2}
$$
for any integer $k>0$.
\end{lemma}

\begin{proof}
    We argue by Moser's method.  Without loss of generality, we may assume that $x_0\in\partial\Omega$. For $m>0$, we let
    \begin{align*}
        v_m(x)=\begin{cases}
            u(x),\quad &\text{if}\ 0<u(x)<m;\\
            0,&\text{otherwise}
        \end{cases}
    \end{align*}
    and $\varphi\in C_0^\infty(B(x_0,r))$. Note that, for $k\geq 1$, set the test function $\theta=v_m^k \varphi^2$ and then
    \begin{align}\label{3.0}
        {\int_{D(x_0,r)}}\langle A\nabla u,\nabla\theta\rangle\,dx \nonumber
        &=-\int_{D(x_0,r)}\langle {\rm div}(A\nabla u),\theta\rangle\,dx+\int_{\partial D(x_0,r)}\frac{\partial u}{\partial \nu}\theta\,d\sigma\\\nonumber
        &=-\int_{D(x_0,r)}\langle Vu,\theta\rangle\,dx+\int_{\Delta(x_0,r) }\frac{\partial u}{\partial \nu}\theta\,d\sigma\\\nonumber
        &\leq 0
    \end{align}
    where we have used the assumption $V>0$ a.e..
    Next, let $\ell=(1+k)/2$, and direct calculation yields that
    \[
    \frac{k}{\ell^2}\int_{D(x_0,r)} \langle A\nabla  v_m^\ell, \nabla v_m^\ell\rangle\varphi^2\,dx+\frac{2}{\ell}\int_{D(x_0,r)} \langle A\nabla v_m^\ell, \nabla\varphi\rangle\varphi v_m^\ell\,dx\leq0.
    \]
    Note that
    \begin{align*}
        \langle A\nabla (v_m^\ell\varphi), \nabla (v_m^\ell\varphi)\rangle=\langle A\nabla  v_m^\ell, \nabla v_m^\ell\rangle\varphi^2+\langle A\nabla \varphi, \nabla \varphi\rangle v_m^{2\ell}+2\langle A\nabla v_m^\ell, \nabla \varphi\rangle v_m^\ell \varphi
    \end{align*}
    where the symmetry of $A$ is used. Then
    \begin{equation}\label{3.1}
     \begin{aligned}
        &\int_{D(x_0,r)} \langle A\nabla(v_m^\ell\varphi), \nabla (v_m^\ell\varphi)\rangle\,dx\\
        &\leq\int_{D(x_0,r)} \langle A\nabla v_m^\ell, \nabla v_m^\ell\rangle\varphi^2\,dx+\int_{D(x_0,r)}\langle A\nabla \varphi, \nabla \varphi\rangle v_m^{2\ell}\,dx+2\int_{D(x_0,r)} \langle A\nabla v_m^\ell, \nabla \varphi\rangle v_m^\ell \varphi\,dx\\
        &\leq\left(2-\frac{2\ell}{k}\right)\int_{D(x_0,r)} \langle A\nabla v_m^\ell, \nabla\varphi\rangle v_m^\ell\varphi\,dx +\int_{D(x_0,r)} \langle A\nabla \varphi, \nabla \varphi\rangle v_m^{2\ell}\,dx\\
        &=\left(1-\frac{1}{k}\right)\int_{D(x_0,r)} \langle A\nabla v_m^\ell, \nabla\varphi\rangle v_m^\ell\varphi\,dx +\int_{D(x_0,r)} \langle A\nabla \varphi, \nabla \varphi\rangle v_m^{2\ell}\,dx.
     \end{aligned}
    \end{equation}
    Again
    \[
    \langle A\nabla(v_m^\ell\varphi),\nabla\varphi\rangle v_m^\ell=\langle A\nabla \varphi, \nabla \varphi\rangle v_m^{2\ell}+\langle A\nabla v_m^\ell, \nabla \varphi\rangle v_m^\ell \varphi.
    \]
    Plug this into \eqref{3.1}, and use Cauchy's inequality, thus we have
    \[
    \int_{D(x_0,r)}|\nabla(v_m^\ell\varphi)|^2\,dx\leq C\left(1+\frac{1}{k^2}\right) \int_{D(x_0,r)} (v_m^\ell|\nabla \varphi|)^2\,dx
    \]
    where \eqref{ellipticity} is used. Next, let $p=\frac{2d}{d-2}$, then Sobolev embedding gives
    \[
    \left(\int_{D(x_0,r)}|v_m^\ell\varphi|^p\,dx\right)^{\frac{2}{p}}\leq C\left(1+\frac{1}{k^2}\right) \int_{D(x_0,r)}(v_m^\ell|\nabla \varphi|)^2\,dx.
    \]

    Next, we use the iteration argument. To handle this, let  $0<s<t<r$, and select $\varphi\in C_0^\infty(B(x_0,r))$, such that $\varphi\equiv1$ in $B(x_0,s)$ and $\varphi\equiv0$ outside $B(x_0,t)$, then
    \[
    \left(\int_{D(x_0,s)}v_m^{\ell p}\,dx\right)^{\frac{1}{\ell p}}\leq \left[C\left(1+\frac{1}{k^2}\right)^{\frac{1}{2}}\frac{1}{t-s}\right]^{\frac{1}{\ell}}
    \left(\int_{D(x_0,t)}v_m^{2\ell}\,dx\right)^{\frac{1}{2\ell}}.
    \]
    Let
    \[
    \gamma:=\frac{d}{d-2}>1.
    \]
    We have proved that for any $q\geq 2$,
    \[
    \left(\int_{D(x_0,s)}v_m^{\gamma q}\,dx\right)^{\frac{1}{\gamma q}}\leq\left[\frac{Cq}{t-s}\right]^{\frac{2}{q}} \left(\int_{D(x_0,t)}v_m^{q}\,dx\right)^{\frac{1}{q}}.
    \]
    This estimate suggests that we iterate, beginning with $q=2,$ as $2,2\gamma, 2\gamma^2,\cdots$. Now set for $i=0,1,2,\cdots$,
    $$
    q_i=2\gamma^i~~~and~~~~r_i=\frac{r}{2}+\frac{r}{2^{i+1}}.
    $$
    Notice that $q_i=\gamma q_{i-1}$ and $r_i-r_{i+1}=\frac{r}{2^{i+2}}$. Thus
    \begin{equation}\label{3.2}
        \left(\int_{D(x_0,r_{i+1})}v_m^{2\gamma^{i+1}}\,dx\right)^{\frac{1}{2\gamma^{i+1}}} \leq\left(\frac{C\gamma^i}{r_i-r_{i+1}}\right)^{\frac{1}{\gamma^i}} \left(\int_{D(x_0,r_{i})}v_m^{2\gamma^i }\,dx\right)^{\frac{1}{2\gamma^i }}.
    \end{equation}
    By iterating \eqref{3.2}
    we obtain
    \begin{equation}
       \left(\int_{D(x_0,r_{i+1})}v_m^{2\gamma^{i+1}}\,dx\right)^{\frac{1}{2\gamma^{i+1}}} \leq \left(\frac{C}{r}\right)^{\sum\limits_{j=0}^{i+1}\frac{1}{\gamma^j}}
        (2\gamma)^{\sum\limits_{j=0}^{i+1}\frac{j}{\gamma^j}} \left(\int_{D(x_0,r)}v_m^{2}\,dx\right)^{\frac{1}{2}},
    \end{equation}
    that is,
    \begin{equation}\label{3.3}
       \left(\int_{D(x_0,r_{i+1})}v_m^{2\gamma^{i+1}}\,dx\right)^{\frac{1}{2\gamma^{i+1}}} \leq \left(\frac{C}{r}\right)^{\frac{d}{2}}
        \left(\frac{2d}{d-2}\right)^{\frac{d(d-2)}{4}} \left(\int_{D(x_0,r)}v_m^{2}\,dx\right)^{\frac{1}{2}}.
    \end{equation}
    Finally taking $i\to\infty$ in \eqref{3.3} yields that
    \begin{equation}\label{3.4}
        \sup\limits_{D(x_0,\frac{r}{2})}v_m\leq C\left(\fint_{D(x_0,r)}v_m^2\,dx\right)^{\frac{1}{2}}.
    \end{equation}
    Let $m\to\infty$ in \eqref{3.4}, and we obtain
    \begin{equation}\label{3.5}
     \sup\limits_{D(x_0,\frac{r}{2})}u^+\leq C\left(\fint_{D(x_0,r)}|u|^2\,dx\right)^{\frac{1}{2}}
    \end{equation}
    where $u^+$ denotes the positive part of $u$. By applying \eqref{3.5} to $-u$, we obtain a similar inequality for $u^-$, the negative part of $u$. With Lemma \ref{Fefferman-Phong} as well as Lemma \ref{lemma-function m} at disposal, we complete the proof.
\end{proof}

\begin{remark}
  Lemma \ref{pointwise estimate} is also valid if $\frac{\partial u}{\partial \nu} =0$ on $\Delta(x_0,r)$ is replaced by $u=0$ on $\Delta(x_0,r)$.
\end{remark}

Next, we give some estimates of the fundamental solution. Let $\Gamma(x,y)$ denote the fundamental solution of the Schr\"odinger operator  $-\text{div}(A \nabla )+ V$ in $\mathbb{R}^d$ and $\widetilde{\Gamma}(x,y)$ denote the fundamental solution of the elliptic operator  $-\text{div}(A \nabla )$ in $\mathbb{R}^d$. Note that
$$0\leq\Gamma(x,y)\leq\widetilde{\Gamma}(x,y)\leq\frac{C}{|x-y|^{d-2}}\quad\text{ for all }x\neq y.$$
In view of the interior Lipschtiz estimates \cite{Dong-2012} for second order elliptic equation $-\text{div}(A \nabla u)=0$ in $B(x,2r)\subset \mathbb{R}^d$ with coefficients satisfying \eqref{ellipticity}-\eqref{symmetric} and \eqref{Dini}, we know that for any $z\in B(x,r)$,
$$
|\nabla u(z)|\leq \frac{C}{r}\fint_{B(x,2r)}|u(y)|dy.
$$
This, together with Lemma \ref{lemma-function m}, Lemma \ref{pointwise estimate} and the fact $0\leq V(x)\leq Cm(x,V)^2$, leads to the estimate of fundamental solution, as follows.
\begin{prop}\label{estimate of fundamental solution-1}
For any $x,y \in \mathbb{R}^d$,
\[
|\Gamma(x,y)|  \leq \frac{C_{k}}{(1+|x-y| m(x, V))^{k}} \cdot \frac{1}{|x-y|^{d-2}},
\]
and
\[
\left|\nabla_x \Gamma(x,y)\right|  \leq \frac{C_{k}}{(1+|x-y| m(x, V))^{k}} \cdot \frac{1}{|x-y|^{d-1}}
\]
where $k \geq 1$ is an arbitrary integer.
\end{prop}

\begin{cor}
  For $x,y\in\mathbb{R}^d$ and $|x-y|\leq \frac{2}{m(x,V)}$,
  \begin{equation}\label{f-solution-1}
\aligned
\left|\Gamma(x,y)-\widetilde{\Gamma}(x,y)\right| \leq
 \begin{cases}
   C m(x, V),\quad  &\text { if } d=3,\\[0.2em]
   C m(x, V)^{2}\log\frac{C}{|x-y|m(x,V)},\quad  &\text { if } d=4,\\[0.2em]
   C m(x, V)^{2}|x-y|^{4-d},\quad  &\text { if } d>4,
 \end{cases}
\endaligned
 \end{equation}
 and
 \begin{equation}\label{f-solution-2}
 \aligned
\left|\nabla_{x} \Gamma(x,y)-\nabla_{x} \widetilde{\Gamma}(x, y)\right| \leq
 \begin{cases}
   C m(x, V)^{2}\log\frac{C}{|x-y|m(x,V)},\quad  &\text { if } d=3,\\[0.2em]
   C m(x, V)^{2}|x-y|^{3-d}, \quad &\text { if } d>3.
 \end{cases}
\endaligned
\end{equation}
\end{cor}


For $R>0$ sufficiently large, set
$$\Omega_R=\{(x',x_d)\in\R^d: |x'|<R\text{ and }\psi(x')<x_d<\psi(x')+R\}. $$
For $x\in \Omega_R$, let $v_R^x(y)$ be the solution of
\begin{align*}
  \begin{cases}
    \mathcal{L} u+V u =0\quad\text{in }\Omega_R,\\[.2em]
    \displaystyle\frac{\partial u}{\partial \nu}=\frac{\partial \Gamma(x,y)}{\partial \nu_y}\quad\text{on }\partial\Omega_R.
  \end{cases}
\end{align*}
Define $N_R(x,y)$ the Neumann function for the Schr\"odinger equation  $\mathcal{L} u+V u =0$ in $\Omega_R$ as
$$N_R(x,y)=\Gamma(x,y)-v_R^x(y),\quad\text{for }x,y\in\Omega_R.$$


 \begin{prop}\label{Nr-estimate}
For any $x,y \in \Omega_R$,
\begin{equation}
  |N_R (x,y)|  \leq \frac{C_{k}}{(1+|x-y| m(y, V))^{k}} \cdot \frac{1}{|x-y|^{d-2}},
\end{equation}
where $k \geq 1$ is an arbitrary integer and $C_{k}$ is independent of $R$.
\end{prop}
\begin{proof}
  This directly follows from Lemma \ref{lemma-function m}, Lemma \ref{Fefferman-Phong} and Lemma \ref{pointwise estimate} as well as a duality argument.
\end{proof}

We are ready to prove Theorem \ref{Holer-estimate}.

\begin{proof}[Proof of Theorem \ref{Holer-estimate}]
We follow the same line of argument as in \cite{Shen-1999}. Let $B=B(x_{0},R)$. We may assume that $\sup _{B}|u|<\infty$.
Let $v(x)=u(x)+\int_B\widetilde{\Gamma}(x,z)V(z)u(z)\,dz$. 
It can be verified that $v\in W^{1,2}(B)$ and $\mathcal{L}v=0$ in $B$. Then by the classical De Giorgi-Nash estimate for $x, y \in B(x_{0}, R / 2)$,
\begin{equation}
  \begin{aligned}
    |v(x)-v(y)|&\leq C\left(\frac{|x-y|}{R}\right)^{\alpha_1} \sup _{B}|v|\\
    &\leq C\left(\frac{|x-y|}{R}\right)^{\alpha_1} \sup _{B}|u|\left\{1+\sup_{x\in B}\int_B\frac{V(z)}{|x-z|^{d-2}}\,dz\right\}\\
    &\leq C\left(\frac{|x-y|}{R}\right)^{\alpha_1} \sup _{B}|u|\left\{1+R^{2}\fint_{B}V\,dy\right\}
  \end{aligned}
\end{equation}
where $\alpha_1\in(0,1)$.
Now, let $r=|x-y|<R$. By Proposition \ref{estimate of fundamental solution-1},
\begin{align*}
  &\left|\int_B(\widetilde{\Gamma}(x,z)-\widetilde{\Gamma}(y,z))V(z)u(z)\,dz\right|
  \leq C\sup _{B}|u|\int_B\left|\frac{V(z)}{|x-z|^{d-2}}- \frac{V(z)}{|z-y|^{d-2}}\right|\,dz.
\end{align*}
Let $B_1=\{z\in B:|z-x|\leq2r\text{ or }|z-y|\leq2r\}$ and $B_2=\{z\in B:|z-x|>2r,|z-y|>2r\}$.
A direct computation leads that there exits $\alpha_2\in(0,1)$ such that
\begin{align*}
  I_1:&=\int_{B_1} \left|\frac{V(z)}{|x-z|^{d-2}}- \frac{V(z)}{|z-y|^{d-2}}\right|\,dz\leq C\int_{B(x,3r)} \frac{V(z)}{|x-z|^{d-2}}\,dz+ C\int_{B(y,3r)} \frac{V(z)}{|z-y|^{d-2}}\,dz\\
  &\leq C\left(\frac{|x-y|}{R}\right)^{\alpha_2}\left\{1+R^{2}\fint_{B}V\,dy\right\},
\end{align*}
and by Proposition \ref{upper-bound-V},
\begin{align*}
I_2:&=\int_{B_2} \left|\frac{V(z)}{|x-z|^{d-2}}- \frac{V(z)}{|z-y|^{d-2}}\right|\,dz
\leq Cr\int_{r<|z-x|<2R}\frac{V(z)}{|x-z|^{d-1}}dz\\&
 \leq Cr\int_{r}^{2R}t^{-d}\int_{B(x,t)}V(z)dz dt\leq Cr\int_{r}^{2R}\left(\frac tR\right)^{\alpha_2} t^{-2}dt\left( R^{2}\fint_{B(x,2R)}V(z)dz\right)\\&
 \leq C\left(\frac{|x-y|}{R}\right)^{\alpha_2} \left\{1+R^{2}\fint_{B}V\,dz\right\}.
\end{align*}
Taking $\alpha=\min\{\alpha_1,\alpha_2\}$, we obtain that for $x, y \in B(x_{0}, R / 2)$, there exists $k_1>0$ such that
\begin{equation}\label{Holer-estimate-1}
\begin{aligned}
  |u(x)-u(y)| &\leqslant C\left(\frac{|x-y|}{R}\right)^{\alpha} \left\{1+R^{2}\fint_{B}V\,dz\right\}\sup _{B}|u|\\
  &\leqslant C\left(\frac{|x-y|}{R}\right)^{\alpha} \{1+Rm(x_0,V)\}^{k_1}\sup _{B}|u|
\end{aligned}
\end{equation}
where Proposition \ref{prop2.6} was used. It suffices for us to show
\begin{equation}\label{3.10}
  \left\{1+R m\left(x_{0}, V\right)\right\}^{k_{1}} \sup _{B}|u| \leqslant C \sup _{B(x_{0}, 2R)}|u|.
\end{equation}

To this end, let $\varphi \in C_{0}^{\infty}\left(B\left(x_{0}, 5 R / 3\right)\right)$ such that $0 \leqslant \varphi \leqslant 1, \varphi=1$ on $B\left(x_{0}, 4 R / 3\right)$, and $|\nabla \varphi| \leqslant C R^{-1},\left|\nabla^{2} \varphi\right| \leqslant CR^{-2}$. Since
$$
(\mathcal{L}+V)(u \varphi)=-2 \langle A\nabla u,\nabla \varphi\rangle+u \mathcal{L} \varphi \quad \text { in } B(x_{0}, 2R),
$$
by H\"older's inequality and Caccioppoli's inequality, for any $x \in B$
$$
\begin{aligned}
 |u(x)| \leqslant &\int_{\mathbb{R}^{d}} \left|\Gamma(x, y)\right|\left|-2 \langle A\nabla u,\nabla \varphi\rangle+u \mathcal{L} \varphi\right| d y\\
 \leqslant& \frac{C}{R} \int_{4 R / 3 \leqslant\left|y-x_{0}\right| \leqslant 5 R / 3} \left|\Gamma(x, y)\right||\nabla u(y)| d y+\frac{C}{R^{2}} \int_{4 R / 3 \leqslant\left|y-x_{0}\right| \leqslant 5 R / 3} \left|\Gamma(x, y)\right||u(y)| d y \\
 \leqslant & \frac{C}{R^{2}}\left\{\int_{4 R / 3 \leqslant\left|y-x_{0}\right| \leqslant 5 R / 3}\left|\Gamma(x, y)\right|^{2} d y\right\}^{1 / 2}\left\{\int_{B\left(x_{0}, 2R\right)}|u(y)|^{2} d y\right\}^{1 / 2}.
\end{aligned}
$$
Using Proposition \ref{estimate of fundamental solution-1}, we obtain that  for $x \in B$ and $k>0$
$$
|u(x)| \leqslant \frac{C}{\left\{1+R m\left(x_{0}, V\right)\right\}^{k}} \sup _{B\left(x_{0}, 2R\right)}|u|.
$$
This, together with \eqref{Holer-estimate-1}, completes the proof.
\end{proof}

\section{\bf Rellich estimates}\label{section-4}

In this section we reduce the solvability of the $L^2$
regularity problems for $\mathcal{L}u+Vu=0$ in Lipschitz graph domains to certain boundary Rellich
estimates.

Let $w_+$ and $w_-$ denote its nontangential limits on $\partial \Omega$, taken from inside $\Omega$ and outside $\overline{\Omega}$, respectively.
\begin{thm}\label{Rellich-lemma-5.5}
Suppose $A$ satisfies \eqref{ellipticity}-\eqref{independent} and $V$ satisfies \eqref{V}. Also assume that $u$ is a weak solution to $\mathcal{L} u+Vu=0$ in $\Omega$ with $(\nabla u)^{*} \in$ $L^{2}(\partial \Omega)$ and $|u(x)|+|x||\nabla u(x)|=$ $O\left(|x|^{2-d}\right)$ as $|x| \to \infty$.  Then
\begin{equation}\label{Rellich-identity-5.5}
  \int_{\partial \Omega}\left|\left(\frac{\partial u}{\partial \nu}\right)_\pm\right|^{2} d\sigma\sim \int_{\partial \Omega}\left|\left(\nabla_{tan} u\right)_\pm\right|^{2} d\sigma+\int_{\partial\Omega}|u_\pm|^2m(x, V)^2\,d\sigma.
\end{equation}
\end{thm}

For simplicity, we write $w_+$ as $w$. Then we have the following Rellich identities.
\begin{lemma}\label{Rellich-lemma-5.5-1}
Suppose that $A$ satisfies \eqref{symmetric} and \eqref{independent}.
Let $u$ be as in Theorem \ref{Rellich-lemma-5.5}.
Then
\begin{equation}\label{Rellich-identity-5.5-1}
    \int_{\partial\Omega}n_d a_{ij}\frac{\partial u}{\partial x_i}  \frac{\partial u}{\partial x_j}\, d\sigma
    =2\int_{\partial\Omega}\frac{\partial u}{\partial x_d}\frac{\partial u}{\partial \nu}\,d\sigma+2\int_{\Omega} \frac{\partial u}{\partial x_d} \mathcal{L}(u)\, dx.
\end{equation}
\end{lemma}

\begin{proof}
By using the divergence theorem we have
\begin{align*}
     &\int_{\partial\Omega}n_d a_{ij}\frac{\partial u}{\partial x_i}  \frac{\partial u}{\partial x_j}\, d\sigma=\int_{\Omega} \frac{\partial}{\partial x_d}\left\{ a_{ij}\frac{\partial u}{\partial x_i}  \frac{\partial u}{\partial x_j}\right\}\, dx\\
    &=\int_{\Omega} \frac{\partial}{\partial x_d}\Big\{ a_{ij}\Big\}\frac{\partial u}{\partial x_i} \frac{\partial u}{\partial x_j}\, dx+\int_{\Omega}  a_{ij}\frac{\partial^2 u}{\partial x_d\partial x_i}\cdot \frac{\partial u}{\partial x_j}\, dx+\int_{\Omega}  a_{ij}\frac{\partial u}{\partial x_i}\cdot \frac{\partial^2 u}{\partial x_d\partial x_j}\, dx\\
    &=2\int_{\Omega} a_{ij} \frac{\partial^2 u}{\partial x_d\partial x_i}\cdot \frac{\partial u}{\partial x_j}\, dx
\end{align*}
where \eqref{symmetric} and \eqref{independent} were used. Applying integration by parts, we have
\begin{equation}
    \begin{aligned}
    \label{1}
    2\int_{\Omega}  a_{ij} \frac{\partial^2 u}{\partial x_d\partial x_i}\cdot \frac{\partial u}{\partial x_j}\, dx
    &=2\int_{\partial\Omega} \frac{\partial u}{\partial x_d} \cdot a_{ij} \frac{\partial u}{\partial x_j}\cdot n_i\,d\sigma-2\int_{\Omega} \frac{\partial u}{\partial x_d}\cdot\frac{\partial}{\partial x_i}\left\{a_{ij} \frac{\partial u}{\partial x_j}\right\}\, dx\\
    &=2\int_{\partial\Omega}\frac{\partial u}{\partial x_d}\frac{\partial u}{\partial \nu}\,d\sigma+2\int_{\Omega} \frac{\partial u}{\partial x_d}  \mathcal{L}(u)\, dx.
    \end{aligned}
\end{equation}
Then we obtain \eqref{Rellich-identity-5.5-1}.
\end{proof}

\begin{lemma}\label{Rellich-lemma-5.5-2}
Under the same assumptions as in Lemma \ref{Rellich-lemma-5.5-1}, we have
\begin{equation}\label{Rellich-identity-5.5-2}
\int_{\partial\Omega}n_d a_{ij}\frac{\partial u}{\partial x_i}  \frac{\partial u}{\partial x_j}\, d\sigma
    =2\int_{\partial\Omega}  a_{ij} \frac{\partial u}{\partial x_j}\left\{n_k\frac{\partial }{\partial x_i}-n_i\frac{\partial }{\partial x_d}\right\} u \, dx-2\int_{\Omega} \frac{\partial u}{\partial x_d} \mathcal{L}(u)\, dx.
\end{equation}
\end{lemma}

\begin{proof}
Let $I$ and $J$ denote the left and right hand sides of (\ref{Rellich-identity-5.5-1}) respectively.
Identity (\ref{Rellich-identity-5.5-2})
follows from (\ref{Rellich-identity-5.5-1}) by writing
$J$ as $2I-J$.
\end{proof}

\begin{cor}\label{Rellich-estimate-1}
Under the same assumptions as in Theorem \ref{Rellich-lemma-5.5}, we have
    \begin{equation}\label{Rellich-identity-5.5-3}
      \int_{\partial \Omega}\left|\nabla u\right|^{2} d\sigma\leq C\int_{\partial \Omega}\left|\frac{\partial u}{\partial \nu}\right|^{2} d\sigma+C\int_{\Omega}|\nabla u||u|m(x, V)^2\,dx
    \end{equation}
    and
    \begin{equation}\label{Rellich-identity-5.5-4}
      \int_{\partial \Omega}\left|\nabla u\right|^{2} d\sigma\leq C\int_{\partial \Omega}\left|\nabla_{tan} u\right|^{2} d\sigma+C\int_{\Omega}|\nabla u||u|m(x, V)^2\,dx
    \end{equation}
where $C$ depends only on $\mu$, $M$ and the constant in \eqref{V}.
\end{cor}
\begin{proof}
This directly follows from Lemma \ref{Rellich-lemma-5.5-1}, \ref{Rellich-lemma-5.5-2} and the fact $V(x)\leq Cm(x,V)^2$.
\end{proof}

Next, we will establish a number of results which essentially from \cite{Shen-1994-IUMJ}.
Other versions of these lemmas appeared in \cite{Shen-1995} and
\cite{Shen-1999}, and are related to the ideas of Fefferman and Phong \cite{Fefferman-1983}. We
omit the proofs in our exposition.
\begin{lemma}\label{lemma-m-1}
 Let $u$ be as in Theorem \ref{Rellich-lemma-5.5}. Then we have the following
 \begin{align}\label{lemma-m-2}
        \int_{\Omega}|u(x)|^{2} m(x, V)^{3} d x \leq C \int_{\partial \Omega}\left|\frac{\partial u}{\partial \nu}\right|^{2} d\sigma,
    \end{align}
    \begin{align}\label{estimate-Rellich-mu3}
        \int_{\Omega}|u(x)|^{2} m(x, V)^{3} d x \leq C \int_{\partial \Omega}|u(x)|^{2} m(x, V)^{2} d \sigma,
    \end{align}
    \begin{align}\label{estimate-app-Rellich-3}
        \int_{\partial\Omega}|u|^2m(x,V)^2\,d\sigma & \leq C\int_{\Omega}|u|^2m(x,V)^3\,dx+C\int_{\Omega}|\nabla u||u|m(x, V)^2\,dx,
    \end{align}
 \begin{align}\label{estimate-Rellich-mu4}
     \int_{\Omega}|\nabla u(x)|^{2} m(x, V)\, d x \leq C \int_{\partial \Omega}\left|\frac{\partial u}{\partial\nu}\right||u|m(x,V)\,d \sigma+C\int_{\Omega}| u(x)|^{2} m(x, V)^3\, d x.
 \end{align}
\end{lemma}

\begin{proof}
This Lemma is essentially the same as Lemmas $2.2$, $2.6$ and $2.7$ in \cite{Shen-1994-IUMJ} with slightly modification. We omit the proof.
\end{proof}

Now we are ready to give the proof of Theorem \ref{Rellich-lemma-5.5}.
\begin{proof}[Proof of Theorem \ref{Rellich-lemma-5.5}]
It follows from \eqref{Rellich-identity-5.5-3} and Cauchy's inequality that
\begin{equation}\label{4.9}
  \begin{aligned}
  \int_{\partial \Omega}\left|\nabla u\right|^{2} d\sigma
  &\leq C\int_{\partial \Omega}\left|\frac{\partial u}{\partial \nu}\right|^{2} d\sigma+C\int_{\Omega}|\nabla u||u|m(x, V)^2\,dx\\
  &\leq C\int_{\partial \Omega}\left|\frac{\partial u}{\partial \nu}\right|^{2} d\sigma+C\int_{\Omega}|\nabla u|^2m(x, V)\,dx+C\int_{\Omega}|u|^2m(x, V)^3\,dx\\
  &\leq C\int_{\partial \Omega}\left|\frac{\partial u}{\partial \nu}\right|^{2} d\sigma+C\int_{\Omega}|\nabla u|^2m(x, V)\,dx
  \end{aligned}
\end{equation}
where \eqref{lemma-m-2} was used in the last inequality. Employing \eqref{estimate-Rellich-mu4}, we obtain
\begin{equation}\label{4.11}
\begin{aligned}
     &\int_{\Omega}|\nabla u(x)|^{2} m(x, V)\, d x
     \leq C \int_{\partial \Omega}\left|\frac{\partial u}{\partial\nu}\right||u|m(x,V)\,d \sigma+C\int_{\Omega}| u(x)|^{2} m(x, V)^3\, d x\\
     &\leq C \int_{\partial \Omega}\left|\frac{\partial u}{\partial\nu}\right|^2\,d \sigma+C\e \int_{\partial \Omega}|u|^2m(x,V)^2\,d \sigma+C\int_{\Omega}| u(x)|^{2} m(x, V)^3\, d x\\
     &\leq C_\e \int_{\partial \Omega}\left|\frac{\partial u}{\partial\nu}\right|^2\,d \sigma+C\e \int_{\partial \Omega}|u|^2m(x,V)^2\,d \sigma
 \end{aligned}
\end{equation}
where we use Cauchy's inequality with an $\e$ in the second inequality and \eqref{lemma-m-2} was used in the last one. To estimate $\int_{\partial\Omega}|u|^2m(x,V)^2\,d\sigma$, we use \eqref{estimate-app-Rellich-3} and H\"older's inequality to obtain
\begin{equation}\label{4.10}
   \begin{aligned}
     \int_{\partial\Omega}|u|^2m(x,V)^2\,d\sigma
     & \leq C\int_{\Omega}|u|^2m(x,V)^3\,dx+C\int_{\Omega}|\nabla u|^2m(x, V)\,dx\\
     &\leq C\int_{\partial \Omega}\left|\frac{\partial u}{\partial \nu}\right|^{2} d\sigma+C\int_{\Omega}|\nabla u|^2m(x, V)\,dx
   \end{aligned}
\end{equation}
where \eqref{lemma-m-2} was used again. Then it deduced from \eqref{4.9}-\eqref{4.11} that
$$\int_{\partial \Omega}\left|\nabla_{tan} u\right|^{2} d\sigma+\int_{\partial\Omega}|u|^2m(x, V)^2\,dx\leq C\int_{\partial \Omega}\left|\frac{\partial u}{\partial \nu}\right|^{2} d\sigma.$$
Similarly, \eqref{Rellich-identity-5.5-4} implies
\begin{equation}\label{4.12}
    \int_{\partial \Omega}\left|\frac{\partial u}{\partial \nu}\right|^{2} d\sigma\leq\int_{\partial \Omega}\left|\nabla u\right|^{2} d\sigma\leq C\int_{\partial \Omega}\left|\nabla_{tan} u\right|^{2} d\sigma+C\int_{\Omega}|\nabla u||u|m(x, V)^2\,dx.
\end{equation}
It follows from \eqref{lemma-m-2} as well as Cauchy's inequality with an $\e$ that
\begin{equation}\label{4.13}
  \begin{aligned}
    \int_{\Omega}|\nabla u||u|m(x, V)^2\,dx
    &\leq C \int_{\partial \Omega}\left|\frac{\partial u}{\partial\nu}\right||u|m(x,V)\,d \sigma+C \int_{\partial \Omega}|u|^2m(x,V)^2\,d \sigma\\
    &\leq C\e \int_{\partial \Omega}\left|\frac{\partial u}{\partial\nu}\right|^2\,d \sigma+C_{\e} \int_{\partial \Omega}|u|^2m(x,V)^2\,d \sigma
  \end{aligned}
\end{equation}
where \eqref{estimate-Rellich-mu4} was used. Hence $$\int_{\partial \Omega}\left|\frac{\partial u}{\partial \nu}\right|^{2} d\sigma\leq C\int_{\partial \Omega}\left|\nabla_{tan} u\right|^{2} d\sigma+C \int_{\partial \Omega}|u|^2m(x,V)^2\,d \sigma.$$
The same argument in $\Omega_-$ leads
\begin{equation*}
  \int_{\partial \Omega}\left|\left(\frac{\partial u}{\partial \nu}\right)_-\right|^{2} d\sigma\sim \int_{\partial \Omega}\left|\left(\nabla_{tan} u\right)_-\right|^{2} d\sigma+\int_{\partial\Omega}|u_-|^2m(x, V)^2\,d\sigma.
\end{equation*}
 The proof is complete.
\end{proof}

Armed with the Rellich estimates, we are ready to give the proof of Theorem \ref{main-thm-Lp-r} for the case $p=2$, by the method of layer potential.

\begin{proof}[Proof of Theorem \ref{main-thm-Lp-r} for the case $p=2$]
The proof follows by the well-known layer potential method. Roughly speaking, let $f\in L^2(\partial\Omega)$, then
\begin{equation}\label{single-layer}
u(x)=\mathcal{S} (f) (x)
=\int_{\partial\Omega} \Gamma (x, y) f(y)\, d\sigma (y)
\end{equation}
satisfies $\mathcal{L} u+Vu=0$ in $\R^d\backslash\partial\Omega$. Set
$$\widetilde{W}^{1,2}(\partial\Omega)=\{f: \|\nabla_{tan}f\|_{L^2(\partial\Omega)}+ \|fm(x,V)\|_{L^2(\partial\Omega)}<\infty\}.$$
As a result, one may establish
the existence of solutions in the $L^2$ regularity problem in $\Omega$ by showing that the operators $S: L^2(\partial\Omega)\to \widetilde{W}^{1,2}(\partial\Omega)$ is invertible. It follows from Proposition \ref{estimate of fundamental solution-1} as well as \eqref{f-solution-1}-\eqref{f-solution-2} that $\|(\nabla u)^*\|_{L^p(\partial\Omega)}\leq C\|f\|_{L^p(\partial\Omega)}$ for $1<p<\infty$. This combining with Theorem \ref{Rellich-lemma-5.5} completes the proof. We omit the details and refer the reader to see \cite{MT-1999-JFA} or \cite{Shen-1994-IUMJ}.
\end{proof}

\section{\bf \texorpdfstring{$L^p$}{Lp} regularity estimate}\label{section-6}
In this section, we are going to show Theorem \ref{main-thm-Lp-r} holds. We also remark here that our result in this section extends the result in \cite{Shen-1994-IUMJ} to $1<p<2+\e$, even in the case of $-\Delta u+V u=0$ in $\Omega$ and $u=g$ on $\partial\Omega$.

Notice that, the proof of Theorem \ref{main-thm-Lp-r} in the case of $p=2$ was given in Section \ref{section-4}. Thus it suffices for us to show that \eqref{rp-estimate} holds for $1< p<2$ and $2< p<2+\e$. And the two cases will be treated separately by different methods. 

\subsection{\texorpdfstring{$L^p$}{Lp} estimates for \texorpdfstring{$1<p<2$}{1<p<2} }
To handle Theorem \ref{main-thm-Lp-r} for the case of $1<p<2$, it suffices for us to assert the existence of solutions for the regularity problem when the data is in $\mathcal{H}_{1,at}^1$. The proof of these results follows the work of Dahlberg and Kenig \cite{Dahlberg-Kenig-1987} (see also \cite{Shen-1994-IUMJ});
thus, we shall give a brief. We first recall the definition of the Hardy spaces $\mathcal{H}_{1,at}^1$.

We say that $a$ is an $\mathcal{H}_{1,at}^1$-atom on $\partial\Omega$ if there exists a ball $B(P,r)$ with $P\in \partial\Omega$ such that $supp(a)\subset B(P,r)\cap \partial\Omega$, $\|\nabla_{tan}a\|_{L^2(\partial\Omega)}\leq |B(P,r)\cap\partial\Omega|^{-1/2}$.
For $f\in L^\frac{d-1}{d-2}(\partial\Omega)$ if $d\geq3$ and $f\in L^2(\partial\Omega)$ if $d=2$, we say $f\in \mathcal{H}^1_{1,at}(\partial\Omega)$ if there exist a sequence of atoms $\{a_j\}$ and a sequence of real numbers $\{\lambda_j\}$ such that $\sum_{j}|\lambda_j|<\infty$ and $f=\sum_j \lambda_j a_j$. The norm in $\mathcal{H}^1_{1,at}(\partial\Omega)$ is given by
$$\|f\|_{\mathcal{H}^1_{1,at}(\partial\Omega)}:=\inf\bigg\{\sum_{j}|\lambda_j|:f=\sum_j \lambda_j a_j\bigg\}.$$

\begin{thm}\label{L2}
Let $a$ be an atom and $u$ a weak solution of
$$
\mathcal{L} u+V u =0, \quad  \text{ in } \Omega, \quad  \text{and}\quad  u =a\quad \text{ on } \partial\Omega
$$
with $(\nabla u)^*\in L^2(\partial\Omega)$. Then we have
\begin{equation}\label{r-estimate-1<p<2-2}
\int_{\partial\Omega}(\nabla u )^*\,d\sigma
\le C,
\end{equation}
where $C$ depends only on $\mu$, $d$, the constant in \eqref{V} and the Lipschitz character of $\Omega$.
\end{thm}

\begin{proof}
Let $a$ be an atom supported in $\Delta(y_0,r)$ with $y_0\in \partial\Omega$.  It follows from H\"older's inequality, we have
\begin{equation} \label{6.12}
\begin{aligned}
  \int_{\Delta(y_0,10r)}|(\nabla u)^*|\,d\sigma
  &\le Cr^{(d-1)/2}\bigg(\int_{\partial\Omega}|(\nabla u)^*|^2\,d\sigma\bigg)^{1/2}\\
  &\leq Cr^{(d-1)/2}\left\{\int_{\partial\Omega}|\nabla_{tan} a|^2 d\sigma +\int_{\partial\Omega}|a|^2m(x,V)^2 \,d\sigma(x)\right\}^{1/2}
\end{aligned}
\end{equation}
where $L^2$ estimate was used. Using a covering argument and Poincar\'e inequality as  well as \eqref{estimate-m1}, we have
$$\int_{\partial\Omega}|a|^2m(x,V)^2 \,d\sigma(x)\leq C\int_{\partial\Omega}|\nabla_{tan} a|^2 d\sigma.$$
Plugging it into \eqref{6.12} gives
$$\int_{\Delta(y_0,10r)}|(\nabla u)^*|\,d\sigma
  \le C$$
where $\|\nabla_{tan}a\|_{L^2(\partial\Omega)}\leq |\Delta(y_0,r)|^{-1/2}$ was used.

In view of the Poisson representation formula, we have
\begin{align} \label{6.12-1}
u(x)&=-\int_{\partial\Omega} \frac{\partial}{\partial \nu}G(x,y) a(y) d\sigma(y).
\end{align}
It follows from H\"oler's inequality that for $x\in \overline{\Omega} \backslash B(y_0, 2r)$,
\begin{equation}\label{6.17}
\begin{aligned}
  |u(x)|&\leq C\|a\|_{L^2(\partial\Omega)}\left(\int_{\Delta(y_0,r)} |\nabla_y G(x,y)|^2 d\sigma(y)\right)^{1/2}\\
  &\leq C\|a\|_{L^2(\partial\Omega)}r^{\frac{d-3}{2}}\sup_{y\in D(y_0,3r)}|G(x,y)| \leq\frac{Cr^{\beta}}{|x-y_0|^{d-2+\beta}}
\end{aligned}
\end{equation}
where $\beta\in(0,1)$. Next we will show that  there exist $C,c>0$ such that if $R\geq cr$,
\begin{equation}\label{6.15}
  \int_{\partial (\Omega\backslash B(y_0, R))}|(\nabla u)^*|^2\,d\sigma\leq \frac{C r^{2\beta}}{R^{d-1+2\beta}}.
\end{equation}
Assume \eqref{6.15} for a moment. For any fix $j\geq2$, there exist $R_j=2^{j-1}r$ such that
$$\Delta (y_0,2^{j+1}r)\backslash\Delta (y_0,2^{j}r)\subset \partial (\Omega\backslash B(y_0, R_j))$$
and by H\"older's inequality
\begin{equation}\label{6.5}
\begin{aligned}
  \int_{2^j r\leq|y-y_0|<2^{j+1} r}(\nabla u)^*\,d\sigma
  &\le C(2^jr)^{(d-1)/2}\bigg(\int_{2^j r\leq|y-y_0|<2^{j+1} r}|(\nabla u)^*|^2\,dx\bigg)^{1/2}\\
  &\le C(2^jr)^{(d-1)/2}\bigg(\int_{\partial (\Omega\backslash B(y_0, R_j))} |(\nabla u)^*|^2\,dx\bigg)^{1/2}\\&
  \leq C\left(\frac {2^jr}{R_j}\right)^{(d-1)/2}\left(\frac r{R_j}\right)^\beta \leq C2^{-j\beta}.
\end{aligned}
\end{equation}
Then \eqref{r-estimate-1<p<2-2} follows from \eqref{6.5} and \eqref{6.12} by summation.

To prove \eqref{6.15}, let $\Omega_{s}=\Omega\backslash B(y_0, sR)$ with $s\in[1,2]$.
By the $L^2$ solvability we have
\begin{equation}\label{6.14}
\begin{aligned}
  \int_{\partial \Omega_{s}}|(\nabla u)^*|^2\,d\sigma &\leq C\int_{\partial \Omega_{s}}(|\nabla_{tan}u|^2+|um(x,V)|^2)\,d\sigma\\
  &= C\int_{\partial B(y_0, sR)\cap\Omega}(|\nabla_{tan}u|^2+|um(x,V)|^2)\,d\sigma.
\end{aligned}
\end{equation}
We integrate both sides of \eqref{6.14} in $s\in[1,2]$ to obtain
\begin{equation}\label{6.6}
\begin{aligned}
  \int_{\partial (\Omega\backslash B(y_0, R))}|(\nabla u)^*|^2\,d\sigma &\leq \frac CR\int_{D(y_0, 2R)\backslash D(y_0, R)}(|\nabla u|^2+|um(x,V)|^2)\,dx\\
  &\leq CR^{d-3}\sup_{D(y_0, 2.1R)\backslash D(y_0, 0.9R)}|u|^2\\
  &\leq \frac{C r^{2\beta}}{R^{d-1+2\beta}}
\end{aligned}
\end{equation}
where Caccioppoli's inequality and a covering argument as well as Lemma \ref{Fefferman-Phong} were used for the second inequality and \eqref{6.17} in the last inequality. Thus we complete the proof.
\end{proof}

To proceed, we need an interpolation result to handle the case $1< p<2$, as follows.
\begin{thm}\label{interpolation}
Suppose $B$ is a sublinear
operator of weak-types $(\mathcal{H}^1, 1)$ and $(2, 2)$ with weak-type norms $M_1$ and $M_2$.
If $1<p <2$, then $B$ is defined on $L^p(\partial\Omega)$ and
$$
\|Bf\|_{L^p(\partial\Omega)}\leq M\|f\|_{L^p(\partial\Omega)},
$$
where $M$ depends on $M_1$ and $M_2$ but is independent of $f\in L^p(\partial\Omega)$.
\end{thm}
\begin{proof}
See \cite{Coifman-Weiss-1977}.
\end{proof}

\begin{proof}[Proof of Theorem \ref{main-thm-Lp-r} in the case of $1<p<2$]
It follows directly from Theorem \ref{L2} and $L^2$ estimates, by using Theorem \ref{interpolation}.
\end{proof}

\subsection{\texorpdfstring{$L^p$}{} estimates for \texorpdfstring{$2< p<2+\e$}{} }

We use the real variable argument developed by Shen \cite{Shen-2007} to treat the proof of Theorem \ref{main-thm-Lp-r} for the case $2< p<2+\e$.

\begin{thm}\label{real variable method}
Let $E\subset \mathbb{R}^d$ be bounded Lipschitz domain and $F\in L^2(E)$. Let $p>2$ and $f\in L^q(E)$ for some $2<q<p$. Suppose that for each ball $B$ with $|B|\leqslant \beta |E|$, there exist two measurable functions $F_B$, $R_B$ on $2B$ such that $|F|\leqslant |F_B|+|R_B|$ on $2B\cap\Omega$,
\begin{equation}\label{rvm2-1}
  \left\{\fint_{2B\cap E}|R_B|^pdx\right\}^{\frac{1}{p}}\leqslant C_1\left\{\left(\fint_{\alpha B\cap E}|F|^2dx\right)^{\frac{1}{2}} +\sup_{B\subset B^{\prime}}\left(\fint_{B^{\prime}\cap E}|f|^2dx\right)^{\frac{1}{2}}\right\}
\end{equation}
and
\begin{equation}\label{rvm2-2}
\fint_{2B\cap E}|F_B|^2dx
\leqslant C_2\sup_{B\subset
B^{\prime}}\fint_{B^{\prime}}|f|^2dx +\sigma \fint_{\alpha
B}|F|^2dx
\end{equation}
where $C_1, C_2>0$ and $0<\beta<1<\alpha$.
 Then, if  $0 \leqslant \sigma<\sigma_0=\sigma_0(C_1,C_2,d,p,q,\alpha, \beta)$, we have
\begin{equation}\label{real result}
\left\{\fint_{E}|F|^qdx\right\}^{\frac{1}{q}}\leqslant
C\left\{\left(\fint_{E}|F|^2dx\right)^{\frac{1}{2}}
+\left(\fint_{E}|f|^qdx\right)^{\frac{1}{q}}\right\},
\end{equation}
where $C>0$ depends only on $C_1,C_2,n,p,q, \alpha, \beta$.
\end{thm}
\begin{proof}
  See \cite[Theorem 4.2.3]{Shen-book}.
\end{proof}

We are ready to give the proof of Theorem \ref{main-thm-Lp-r} in the case of $2<p<2+\e$.
\begin{proof}[Proof of Theorem \ref{main-thm-Lp-r} in the case of $2<p<2+\e$]
Recall that, for $R>0$ large sufficiently,
$$\Omega_R=\{(x',x_d)\in\R^d: |x'|<R\text{ and }\psi(x')<x_d<\psi(x')+R\}. $$
 It suffices for us to show for $2<p<2+\e$
 \begin{equation}\label{5.12}
    \|(\nabla u)^*\|_{L^p(\partial\Omega\cap\partial\Omega_R)}\leq C\|\nabla_{tan}g_R\|_{L^p(\partial\Omega\cap\partial\Omega_R)}+ C\|g_Rm(x,V)\|_{L^p(\partial\Omega\cap\partial\Omega_R)}
  \end{equation}
  where \begin{align*}
  g_R(x)=\begin{cases}
    g(x),\quad& x\in\partial \Omega\cap\partial \Omega_R,\\[0.1cm]
    0, &\text{otherwise}.
  \end{cases}
\end{align*}
Fix $R$ and a ball $B(x_0,r)$ in $\partial\Omega$. Choose $\varphi \in C_{0}^{\infty}(B(x_0,8r))$ such that
$$
\varphi(x)= \begin{cases}1, & x\in B(x_0,4r) ,\\[0.1cm]
 0, & x\notin B(x_0,8r).\end{cases}
$$
Let $h=\fint_{\Delta(x_0,8r)}g_R\,d\sigma$ and write $u=v+w-h$ where $v$ is a weak solution of $\mathcal{L} v+V v =0$ with boundary data $v=(g_R-h)\varphi$ on $\partial\Omega\cap\partial \Omega_R$. To proceed, we apply Theorem \ref{real variable method} with $F=(\nabla u)^*$, $F_B=(\nabla v)^*$ and $R_B=(\nabla w)^*$. Let $\Delta'(x,r)=\Delta(x,r)\cap\partial \Omega_R$. It can be verified that $F\leq F_B+R_B$ and by $L^2$ solvability of the regularity problem,
\begin{align*}
&\int_{\Delta'(x_0,r)}|F_B|^2\,d\sigma\leq \int_{\partial\Omega\cap\partial\Omega_R}|(\nabla v)^*|^2\,d\sigma\\
&\leq \int_{\partial\Omega\cap\partial\Omega_R}|\nabla_{tan} [(g_R-h)\varphi]|^2\,d\sigma+ \int_{\partial\Omega\cap\partial\Omega_R}|(g_R-h)\varphi|^2m(x,V)^2\,d\sigma\\
&\leq C\int_{\Delta'(x_0,8r)}|\nabla_{tan}g_R|^2\,d\sigma+ C\int_{\Delta'(x_0,8r)}|g_R|^2m(x,V)^2\,d\sigma,
\end{align*}
where Poincar\'e inequality and Lemma \ref{lemma-m-1} were used. To estimate $R_B$, we note that $w=(g_R-h)(1-\varphi)$ on $\partial\Omega\cap\partial\Omega_R$, that is, $w=0$ on $\Delta'(x_0,4r)$. We assume that for some $q>2$,
\begin{equation}\label{rh2}
  \left\{\fint_{\Delta'(x_0,r)}|(\nabla w)^*|^qd\sigma\right\}^{1/q}\leq C\left\{\fint_{\Delta'(x_0,4r)}|(\nabla w)^*|^2d\sigma\right\}^{1/2}
\end{equation}
for a moment. This gives
\begin{align*}
  &\left(\fint_{\Delta'(x_0,r)} |R_B|^q d\sigma\right)^{1/q}\leq C\left(\fint_{\Delta'(x_0,4r)} |R_B|^2 d\sigma\right)^{1/2}\\
  &\leq C\left(\fint_{\Delta'(x_0,4r)} |(\nabla u)^*|^2 d\sigma\right)^{1/2}+C\left(\fint_{\Delta'(x_0,8r)} \left(|\nabla_{tan}g_R|^2+|g_R|^2m(x,V)^2 \right)d\sigma\right)^{1/2}.
\end{align*}
Thus, by Theorem \ref{real variable method}, we obtain \eqref{5.12} for any $2<p<q$.

Next to see \eqref{rh2}, we use $L^2$ estimate for regularity problem to get
\begin{align*}
  &\left(\fint_{\Delta'(x_0,r)} |(\nabla w)^*|^2 d\sigma\right)^{1/2}\leq C\left(\fint_{D(x_0,3r)\cap\Omega_R} |\nabla w|^2 dx\right)^{1/2}\\
  &\leq \frac Cr\left(\fint_{D(x_0,3.1r)\cap\Omega_R} \big|w-\fint_{D(x_0,4r)\cap\Omega_R}w\big|^2 dx\right)^{1/2}\\
  &\leq  C\left(\fint_{D(x_0,4r)\cap\Omega_R} |\nabla w|^{\frac{2d}{d+2}} dx\right)^{\frac{d+2}{2d}} \leq C\left(\fint_{\Delta'(x_0,4r)} |(\nabla w)^*|^{\frac{2d}{d+2}} d\sigma\right)^{\frac{d+2}{2d}}
\end{align*}
where we have used Cacciopoli's inequality and Sobolev-Poincar\'e inequality. This, together with the self-improving property, gives \eqref{rh2} and completes the proof.
\end{proof}

\section{\bf \texorpdfstring{$W^{1,p}$}{W1p} estimate}\label{section-7}

In this section, we are going to prove Theorem \ref{main-thm-W1p} holds. We use a real variable method and a duality argument to prove $\|\nabla u\|_{L^p(\Omega)}\leq C\|f\|_{L^p(\Omega)}$ with $G=0$ for the cases $2<p<3+\e$ and $\frac32-\e<p<2$ respectively. Again by a duality argument, $\|\nabla u\|_{L^p(\Omega)}\leq C\|f\|_{L^p(\Omega)}+C\|G\|_{B^{-1/p,p}(\partial\Omega)}$ for $\frac32-\e<p<3+\e$. Also estimates for the Neumann function are used to obtain $\|V^{\frac12} u\|_{L^p(\Omega)}\leq C\|f\|_{L^p(\Omega)}+C\|G\|_{B^{-1/p,p}(\partial\Omega)}$.

\subsection{\texorpdfstring{$W^{1,p}$}{W1p} estimate for \texorpdfstring{$2<p<3+\e$}{}}

\begin{thm}\label{thm-W1p-1}
  Suppose $A$ satisfies \eqref{ellipticity}-\eqref{Dini} and $V(x)>0$ a.e. satisfies \eqref{V}. Let $f\in L^{p}(\Omega)$. Then there exists $\e>0$ so that if $2<p<3+\e$, the weak solutions to \eqref{W1p} with $G=0$ satisfies the estimate
$$\|\nabla u\|_{L^p(\Omega)}\leq C\|f\|_{L^p(\Omega)},$$
where $C$ depends only on $\mu$, $d$, $p$ and the Lipschitz character of $\Omega$.
\end{thm}

\begin{thm}\label{thm1}
  Let $p>2$. Suppose $A$ satisfy \eqref{ellipticity}-\eqref{symmetric} and $V(x)>0$ a.e. in $\mathbb{R}^d$. Assume that for any ball $B(x_0,r)$ with the property that either $x_0\in \partial\Omega$ or $B(x_0,2r)\subset\Omega$, the weak reverse H\"older inequality
  \begin{equation}
    \label{rh inequality}
    \bigg(\fint_{D(x_0,r)}|\nabla v|^p\bigg)^{\frac{1}{p}}\leq C_0\bigg(\fint_{D(x_0,2r)}|\nabla v|^2\bigg)^{\frac{1}{2}}
  \end{equation}
  holds, whenever $v\in W^{1,2}(D(x_0,2r))$ with $x_0\in\partial\Omega$ satisfies
  \begin{equation}\label{Equation v2}
   \mathcal{L} v+V v =0 \ \text{ in } D(x_0,2r)\quad\text{and}\quad
   \frac{\partial v}{\partial \nu}  =0 \ \text{ on } \Delta(x_0,2r).
  \end{equation}
  Let $u\in W^{1,2}(\Omega)$ be a weak solution of \eqref{W1p} with $f\in L^p(\Omega,\mathbb{R}^d)$ and $G=0$. Then $u\in W^{1,p}(\Omega)$ and
  \begin{equation}
    \label{result1}
    \|\nabla u\|_{L^p(\Omega)}\leq C\|f\|_{L^p(\Omega)},
  \end{equation}
  where $C$ depends only on $d,p,\mu,C_0$ and the Lipschitz character of $\Omega$.
\end{thm}

\begin{proof}
For $R$ large, let \begin{align*}
  f_R(x)=\begin{cases}
    f(x),\quad& x\in\Omega_R\text{ or }x\in\partial \Omega\cap\partial \Omega_R,\\[0.1cm]
    0, &\text{otherwise}.
  \end{cases}
\end{align*}
 It suffices for us to show \begin{equation}
    \|\nabla u\|_{L^p(\Omega_R)}\leq C\|f_R\|_{L^p(\Omega_R)}
  \end{equation}
  where $u$ satisfies
  \begin{equation}
\mathcal{L} u+V u ={\rm div} f_R\quad  \text{in } \Omega_R,\quad\text{ and }\quad\frac{\partial u}{\partial \nu}  =-f_R\cdot n \quad\text{on } \partial\Omega_R.
\end{equation}
Fix $0<r<R/8$ and $B(x,r)$ such that $|B(x,r)| \leqslant \beta|\Omega_R|$ and either $B(x,4r) \subset \Omega_R$ or $B(x,r)$ centers on $\partial \Omega\cap\partial \Omega_R$. Let $\varphi \in C_{0}^{\infty}(B(x,8r))$ be a cut-off function such that $\varphi(x)=1$ in $B(x,4r)$ and $\varphi(x)=0$ outside $B(x,8r)$. We decompose $u=v+w$, where $v, w$ satisfy
\begin{equation}\label{Equation v}
\mathcal{L}v+Vv=\operatorname{div}(\varphi f_R) \quad \text {in } \Omega_R, \quad\text{and}\quad\frac{\partial v}{\partial \nu}=-\varphi f_R \cdot n \quad \text {on } \partial \Omega_R
\end{equation}
and
\begin{equation}\label{Equation w}
\mathcal{L}w+Vw=\operatorname{div}((1-\varphi) f_R) \quad \text {in } \Omega_R, \quad\text{and}\quad\frac{\partial w}{\partial \nu}=(\varphi-1) f_R \cdot n \quad \text {on } \partial \Omega_R
\end{equation}
respectively.

We first claim that \eqref{rvm2-1} and \eqref{rvm2-2} hold for $F=|\nabla u|, F_{B}=|\nabla v|$ and $R_{B}=|\nabla w|$ and postpone the proof to the end of this Theorem. Then
note that $|F| \leqslant\left|F_{B}\right|+\left|R_{B}\right|$, applying Theorem \ref{real variable method}, we obtain
$$
\left\{ \fint_{\Omega_R}|\nabla u|^{q} d x\right\}^{\frac{1}{q}} \leqslant C\left\{\left( \fint_{\Omega_R}|\nabla u|^{2} d x\right)^{\frac{1}{2}}+\left( \fint_{\Omega_R}|f_R|^{q} d x\right)^{\frac{1}{q}}\right\}
$$
for any $2<q<p$. It follows from the self-improving property of the reverse H\"older condition that the above estimate indeed holds for any $2<q<\bar{p}$ with $\bar{p}>p$ and particularly for $q=p$,
\begin{equation}
  \label{rh1}
  \left\{ \fint_{\Omega_R}|\nabla u|^{p} d x\right\}^{\frac{1}{p}} \leqslant C\left\{\left( \fint_{\Omega_R}|\nabla u|^{2} d x\right)^{\frac{1}{2}}+\left( \fint_{\Omega_R}|f_R|^{p} d x\right)^{\frac{1}{p}}\right\}.
\end{equation}
Integrating \eqref{NP} by parts and using H\"older's inequality, we have
$$\fint_{\Omega_R}|\nabla u|^{2}\,dx\leq C\fint_{\Omega_R}|f_R |^{2}\, dx\leq C\left(\fint_{\Omega_R}|f_R |^{p}\right)^{\frac{2}{p}}.$$
where we have also used  the non-negativity of $V$. This, combining with \eqref{rh1}, gives \eqref{result1}.

To see \eqref{rvm2-2}, apply integration by parts to equation \eqref{Equation v}, then we obtain by  the non-negativity of $V(x)$,
\begin{equation}
\label{rh inequality4}
\int_{\Omega_R}|\nabla v|^{2}\,dx\leq C\int_{\Omega_R}|f_R \varphi|^{2}\, dx
\end{equation}
where H\"older's inequality was used. Let $D'(x,r)=B(x,r)\cap\Omega_R$. It follows from \eqref{rh inequality4} that
$$
 \fint_{D'(x,2r)}\left|F_{B}\right|^{2}\, d x \leqslant C\fint_{D'(x,4r)}|\nabla v|^{2}\, d x\leqslant \frac{C}{|D'(x,4r)|} \int_{\Omega_R}|\nabla v|^{2}\, d x\leqslant C \fint_{D'(x,8r)}|f_R|^{2}\, d x .
$$

Next, to verify \eqref{rvm2-1}, note that by \eqref{Equation w}, $w$ satisfies
\begin{equation}\label{Equation w2}
\mathcal{L}w+Vw=0\quad \text {in } D'(x,4r)\quad \text{ and }\quad \displaystyle\frac{\partial w}{\partial \nu}=0\quad  \text{on } \Delta(x,4r)\cap\partial\Omega_R.
\end{equation}
Thus by the reverse H\"older condition \eqref{rh inequality}, we have
$$
\left\{ \fint_{D'(x,2r)}|\nabla w|^{p} d x\right\}^{\frac{1}{p}} \leqslant C\left\{ \fint_{D'(x,4r)}|\nabla w|^{2} d x\right\}^{\frac{1}{2}}.
$$
Again use \eqref{rh inequality4}, then we obtain
$$
\begin{aligned}
\left\{ \fint_{D'(x,2r)}\left|R_{B}\right|^{p} d x\right\}^{\frac{1}{p}}&\leqslant C\left\{ \fint_{D'(x,4r)}|\nabla w|^{2} d x\right\}^{\frac{1}{2}} \\
&\leqslant C\left\{ \fint_{D'(x,4r)}|\nabla u|^{2} d x\right\}^{\frac{1}{2}} +C\left\{ \fint_{D'(x,4r)}|\nabla v|^{2} d x\right\}^{\frac{1}{2}} \\
&\leqslant C\left\{ \fint_{D'(x,4r)}|F|^{2} d x\right\}^{\frac{1}{2}} +C\left\{ \fint_{D'(x,4r)}|f_R|^{2} d x\right\}^{\frac{1}{2}}
\end{aligned}
$$
which gives \eqref{rvm2-1}. This completes the proof.
\end{proof}

Denote
$$Z_r=\left\{\left(x^{\prime}, x_{d}\right):\left|x^{\prime}\right|<r, \Psi\left(x^{\prime}\right)<x_{d}<10 \sqrt{d}(M+1) r\right\}$$
and
$$S_r=\left\{\left(x^{\prime}, \Psi\left(x^{\prime}\right)\right):\left|x^{\prime}\right|<r\right\},$$
where $\Psi: \mathbb{R}^{n-1} \rightarrow \mathbb{R}$ is a Lipschitz function with $\Psi(0)=0$ and $\|\nabla \Psi\|_{\infty}\leq M$.

\begin{lemma}\label{rh-2<p<3+e}
Suppose $A$ satisfies \eqref{ellipticity}-\eqref{Dini} and $V(x)>0$ a.e. satisfies \eqref{V}. Let $u\in W^{1,2}(Z_{2r} )$ be a weak solution to $\mathcal{L}u+Vu=0$ in $Z_{2 r}$ and $\frac{\partial u}{\partial \nu}=0$ on $S_{2 r}$. Then there exists $\e>0$,  for  $2<q<3+\e$,
\begin{equation}\label{rh inequality3}
  \left\{\frac{1}{r^{d}} \int_{Z_{r}}|\nabla u|^{q} d x\right\}^{\frac{1}{q}} \leqslant C\left\{\frac{1}{r^{d}} \int_{Z_{3r}}|\nabla u|^{2} d x\right\}^{\frac{1}{2}}.
\end{equation}
where $C$ depends only on $d, \mu$ and the Lipschitz character of $\Omega$..
\end{lemma}

\begin{proof}
  It follows from Lemma \ref{Holer-estimate} that
  for all $x \in Z_{r}$,
\begin{equation}
  \label{Holder}
  |\nabla u(x)| \leqslant C\left(x_{d}-\Psi\left(x^{\prime}\right)\right)^{\alpha-1} \frac{1}{r^{\alpha-1}}\left\{\frac{1}{r^{d}} \int_{Z_{2r}}|\nabla u|^{2} d y\right\}^{\frac{1}{2}}
\end{equation}
where $\alpha\in(0,1)$.
Next, note that for $q>2$,
\begin{equation}\label{14}
  \begin{aligned}
&\int_{Z_{r}}|\nabla u|^{q} d x =\int_{Z_{r}}|\nabla u|^{2}|\nabla u|^{q-2} d x \\
\leqslant & \frac{C}{r^{(q-2)(\alpha-1)}} \int_{Z_{r}}|\nabla u|^{2}\left(x_{d}-\Psi\left(x^{\prime}\right)\right)^{(q-2)(\alpha-1)} d x \left\{\frac{1}{r^{d}} \int_{Z_{2r}}|\nabla u|^{2} d x\right\}^{\frac{q}{2}-1},
\end{aligned}
\end{equation}
is valid where \eqref{Holder} was employed. Let
$$
(\nabla u)_{r}^{*}\left(x^{\prime}, \Psi\left(x^{\prime}\right)\right)=\sup \left\{\left|\nabla u\left(x^{\prime}, x_{d}\right)\right|: \Psi\left(x^{\prime}\right)<x_{d}<C_1 r\right\} .
$$
Then, if $(q-2)(\alpha-1)>-1$, employing polar coordinates gives
\begin{equation}\label{15}
\begin{aligned}
 \frac{C}{r^{(q-2)(\alpha-1)}} \int_{S_{r}}\left|(\nabla u)_{r}^{*}\right|^{2}\, d \sigma \int_{0}^{Cr} t^{(q-2)(\alpha-1)}\, d t \leqslant C r \int_{S_{r}}\left|(\nabla u)_{r}^{*}\right|^{2}\, d \sigma.
\end{aligned}
\end{equation}
Thus by $L^2$ estimate, Lemma \ref{main-thm-Lp-n}, for $t\in[1,2]$,
\begin{align*}
\int_{S_{r}}\left|(\nabla u)_{r}^{*}\right|^{2} d \sigma &\leqslant \int_{\partial Z_{t r}}\left|(\nabla u)^{*}\right|^{2} d \sigma \leqslant  C \int_{\partial Z_{t r}}\left|\frac{\partial u}{\partial \nu}\right|^{2} d x.
\end{align*}
Integrating with respect to $t\in[1,2]$, we obtain
\begin{equation}
  \label{17}
  \int_{S_{r}}\left|(\nabla u)_{r}^{*}\right|^{2} d \sigma \leqslant \frac{C}{r} \int_{Z_{2r}}|\nabla u|^{2} d x .
\end{equation}
This, combining with \eqref{14} and \eqref{15}, gives
$$\int_{Z_{r}}|\nabla u|^{q} d x\leqslant C r \int_{S_{r}}\left|(\nabla u)_{r}^{*}\right|^{2}\, d \sigma\left\{\frac{1}{r^{d}} \int_{Z_{2r}}|\nabla u|^{2}\, d y\right\}^{\frac{q}{2}-1}\leqslant C\left(\frac{1}{r^d}\right)^{\frac{q}{2}-1} \left\{\int_{Z_{2r}}|\nabla u|^{2} d x\right\}^{\frac{q}{2}}$$
that is,
$$
\left(\frac{1}{r^{d}} \int_{Z_{r}}|\nabla u|^{q} d x\right)^{\frac{1}{q}} \leqslant C\left(\frac{1}{r^{d}} \int_{Z_{2r}}|\nabla u|^{2} d x\right)^{\frac{1}{2}}
$$
for $2<q<3+\frac{\alpha}{1-\alpha}$.
\end{proof}

\begin{proof}[Proof of Theorem \ref{thm-W1p-1}]
  By Theorem \ref{thm1}, it suffices for us to show for any ball $B(x_0,r)$ with the property that either $x_0\in \partial\Omega$ or $B(x_0,2r)\subset\Omega$ for some $r>0$, the weak solutions of \eqref{Equation v2} satisfies the weak reverse H\"older inequality \eqref{rh inequality}. If  $B(x_0, 2 r) \subset \Omega$, the desired result follows from the interior estimate. If $x_0\in \partial\Omega$, Lemma \ref{rh-2<p<3+e} gives \eqref{rh inequality}. The proof is completed.
\end{proof}


\subsection{\texorpdfstring{$W^{1,p}$}{} estimate for \texorpdfstring{$\frac{3}{2}-\e<p<2$}{}}

In the case $\frac32-\e<p<2$, we use a duality argument.
\begin{thm}\label{thm-W1p-2}
  Suppose $A$ satisfies \eqref{ellipticity}-\eqref{Dini} and $V(x)>0$ a.e. satisfies \eqref{V}. Let $f\in L^{p}(\Omega)$. Then there exists $\e>0$ so that if $\frac32-\e<p<2$, the weak solutions to \eqref{W1p} with $G=0$ satisfies the estimate
$$\|\nabla u\|_{L^p(\Omega)}\leq C\|f\|_{L^p(\Omega)},$$
where $C$ depends only on $\mu$, $d$, $p$ and the Lipschitz character of $\Omega$.
\end{thm}

\begin{proof}
  Let $h\in C_c^\infty(\Omega)$ and $v$ be a weak solution of $\mathcal{L}v+Vv={\rm div} h$ in $\Omega$ and $\frac{\partial v}{\partial \nu}=0$ on $\partial\Omega$. Let $q=\frac{p}{p-1}$. The weak formulations of variational solution of $u$ and $u$ imply that
$$\left|\int_\Omega h\cdot \nabla u\,dx\right|=\left|\int_\Omega f\cdot \nabla v\,dx\right|.$$
It follows from H\"older inequality and the result for $q\in(2,3+\e)$, Theorem \ref{thm-W1p-1} that
$$\|\nabla u\|_{L^p(\Omega)}=\sup_{\|h\|_{L^q(\Omega)}\leq1}\left|\int_\Omega h\cdot \nabla u\,dx\right|\leq C\|f\|_{L^p(\Omega)}.$$
\end{proof}

To show Theorem \ref{main-thm-W1p} we need the following estimate for Neumann function.
 \begin{prop}\label{Neumann function estimate}
For any $x,y \in \Omega$,
\begin{equation}
  |N (x,y)|  \leq \frac{C_{k}}{(1+|x-y| m(y, V))^{k}} \cdot \frac{1}{|x-y|^{d-2}},
\end{equation}
where $k \geq 1$ is an arbitrary integer.
\end{prop}
\begin{proof}
  The same argument as in the proof of Proposition \ref{Nr-estimate} completes the proof.
\end{proof}

Finally we are in a position to give the proof of Theorem \ref{main-thm-W1p}.
\begin{proof}[Proof of Theorem \ref{main-thm-W1p}]
Decompose $u=u_1+u_2$ where $u_1,u_2$ are weak solutions of
\begin{equation*}
  \begin{aligned}
  \begin{cases}
\mathcal{L} u_1+Vu_1 ={\rm div} f  &\text{in } \Omega,\\[0.1cm]
\hspace{3.2em}\frac{\partial u_1}{\partial \nu}  =-f\cdot n&\text{on } \partial\Omega,
  \end{cases}
\end{aligned}
\quad\text{ and }\quad
\begin{aligned}
\begin{cases}
\mathcal{L} u_2+Vu_2 =0\quad  &\text{in } \Omega,\\[0.1cm]
\hspace{3.2em}\frac{\partial u_2}{\partial \nu}  =G\ &\text{on } \partial\Omega.
  \end{cases}
\end{aligned}
\end{equation*}
By using integration by parts and Cauchy's inequality, we have
$$\|\nabla u_1\|_{L^2(\Omega)}\leq C\|f\|_{L^2(\Omega)}.$$
This, combining with Theorem \ref{thm-W1p-1} and Theorem \ref{thm-W1p-2}, gives that there exists $\e>0$ such that for $\frac32-\e<p<3+\e$,
\begin{equation}\label{6.1}
  \|\nabla u_1\|_{L^p(\Omega)}\leq C\|f\|_{L^p(\Omega)}.
\end{equation}

Let $h\in C^\infty_0(\Omega)$ and $v$ be the weak solution to
$$\mathcal{L} (v-c)+V(v-c) ={\rm div}\;h\quad  \text{in } \Omega,\quad\text{ and }\quad\frac{\partial v}{\partial \nu}  =0 \quad\text{on } \partial\Omega,$$
where $c=\fint_{\Omega}v\,dx$. Then the weak formulation, the Sobolev embedding and Poincar\'e inequality imply that
\begin{align*}
  \left|\int_\Omega h\cdot\nabla u_2\, dx\right|=\left|\int_{\partial\Omega} G(v-c)\, dx\right| &\leq\|G\|_{B^{-\frac 1p,p}(\partial\Omega)} \|v-c\|_{B^{\frac 1p,q}(\partial\Omega)}\\
  &\leq\|G\|_{B^{-\frac 1p,p}(\partial\Omega)} \|v-c\|_{W^{1,q}(\Omega)}\\
  &\leq\|G\|_{B^{-\frac 1p,p}(\partial\Omega)} \|\nabla v\|_{L^{q}(\Omega)}
\end{align*}
where $p, q$ are conjugate. It follows from \eqref{6.1} that for $\frac32-\e<q<3+\e$, $$\|\nabla v\|_{L^{q}(\Omega)}=\|\nabla (v-c)\|_{L^{q}(\Omega)}\leq C\|h\|_{L^{q}(\Omega)}.$$
This gives
\begin{equation}\label{6.2}
  \|\nabla u_2\|_{L^p(\Omega)}=\sup_{\|h\|_{L^{q}(\Omega)}\leq 1}\left|\int_\Omega h\cdot\nabla u_2\, dx\right|\leq C\|G\|_{B^{-\frac 1p,p}(\partial\Omega)}.
\end{equation}
In view of \eqref{6.1} and \eqref{6.2}, we obtain that for $\frac32-\e<p<3+\e$,
\begin{equation}\label{6.10}
  \|\nabla u\|_{L^p(\Omega)}\leq C\left\{\|f\|_{L^p(\Omega)}+\|G\|_{B^{-\frac 1p,p}(\partial\Omega)}\right\}.
\end{equation}

It suffices for us to show
\begin{equation}\label{Vu-estimate}
  \|V^\frac{1}{2}u\|_{L^p(\Omega)}\leq C\left\{\|f\|_{L^p(\Omega)}+\|G\|_{B^{-\frac 1p,p}(\partial\Omega)}\right\}.
\end{equation}
It follows from the Poisson representation formula and integration by parts that
\begin{align*}
  u_1(x)&= \int_{\partial\Omega}  N(x, y)\frac{\partial u_1}{\partial \nu}d \sigma(y)+ \int_{\Omega}  N(x, y)(\mathcal{L}+V)u_1\,d y=
  -\int_{\Omega} \nabla_y N(x, y)f(y)d y.
\end{align*}
By H\"older's inequality we have
\begin{equation}\label{4.6}
\begin{aligned}
|u_1(x)| &\leq
\int_{\Omega} |\nabla_y N(x, y)| |f(y)| d y\\
& \leq
 \left\{\int_{\Omega}|\nabla_y N(x, y)| d y\right\}^{1 /q}\left\{\int_{\Omega}|\nabla_y N(x, y)||f(y)|^{p} d y\right\}^{1 / p}
\end{aligned}
\end{equation}
where $q=\frac{p}{p-1}$. Fix $x\in \partial\Omega$. Let $r_0=\frac{1}{m(x,V)}$ and $E_j=\{y\in\Omega: |x-y|\sim2^jr_0\}$. It follows from H\"older's inequality and Caccippoli's inequality that
\begin{equation}\label{N-estimate-1}
  \begin{aligned}
    \int_{E_j}|\nabla_y N(x, y)|\, d y & \leq C(2^jr_0)^{\frac d2}\left(\int_{E_j}|\nabla_y N(x, y)|^2\, d y\right)^{1/2}  \\
    & \leq C(2^jr_0)^{\frac d2-1}\left(\int_{E_j}|N(x, y)|^2\, d y\right)^{1/2}\\
    & \leq C(2^jr_0)^{\frac d2-1}\cdot\frac{(2^jr_0)^{\frac d2}}{(1+2^j)^{k}(2^jr_0)^{d-2}}\\
    &=\frac{C 2^jr_0}{(1+2^j)^{k}}
  \end{aligned}
\end{equation}
where Proposition \ref{Neumann function estimate} was also used. Taking $k=2$ in \eqref{N-estimate-1}, we have
\begin{equation*}\label{6.4}
\begin{aligned}
\int_{\Omega}|\nabla_y N(x, y)|\, d y & \leq  \frac{C}{m(x,V)}\sum_{j=-\infty}^{\infty}\frac{2^j}{(1+2^j)^2}\leq \frac{C}{m(x, V)}.
\end{aligned}
\end{equation*}
Plugging this into \eqref{4.6}, we obtain
$$
\begin{aligned}
|u_1(x)| & \leq \frac{C}{m(x, V)^{1 / q}}
\left(\int_{\Omega}|\nabla_y N(x, y) \| f(y)|^{p} d y\right)^{1 / p}
\end{aligned}
$$
This combining with Proposition \ref{upper-bound-V} gives that
\begin{align*}
  &\int_{\Omega}|V^{\frac12}(x)u_1(x)|^{p}d x\leq C \int_{\Omega}\left|m(x, V)u_1\right|^{p} d x
  \leq C \int_{\Omega}|f(y)|^{p}\left\{\int_{\Omega} m(x, V)|\nabla_y N(x, y)| d x\right\} d y.
\end{align*}
For fixed $y\in\partial\Omega$, Let $r_1=\frac{1}{m(y,V)}$ and $F_j=\{x\in\Omega: |x-y|\sim2^jr_1\}$. Together Lemma \ref{lemma-function m} with \eqref{N-estimate-1} yields that
$$
\begin{aligned}
\int_{F_j}|\nabla_y N(x, y)|m(x, V)\, d x & \leq  \frac{C 2^j r_1}{(1+2^j)^{k}}\cdot(1+2^j)^{k_0}r_1^{-1}\leq \frac{C 2^j}{(1+2^j)^{2}}
\end{aligned}
$$
where $k$ is chosen to be $k_0+2$ in the second inequality. Thus we have
\begin{equation}\label{6.8}
\int_{\Omega} m(x, V)|\nabla_y N(x, y)| d x\leq C\sum_{j=-\infty}^{\infty}\frac{2^j}{(1+2^j)^{2}}\leq C
\end{equation}
which implies for $1<p<\infty$, $\|V^{\frac12}u_1\|_{L^p(\Omega)}\leq C\|f\|_{L^p(\Omega)}$ holds.

It remains to estimate $\|V^{\frac12}u_2\|_{L^p(\Omega)}$. Let $F\in C_0^\infty(\Omega)$ and $v$ solves
\begin{align*}
  \begin{cases}
\mathcal{L} v+Vv =F  &\text{in } \Omega,\\[0.1cm]
\hspace{2.8em}\frac{\partial v}{\partial \nu}  =0&\text{on } \partial\Omega.
  \end{cases}
\end{align*}
Then
\begin{align*}
  \left|\int_\Omega u_2 F \, dx\right|=\left|\int_{\partial\Omega} Gv\, d\sigma\right|\leq \|G\|_{B^{-\frac 1p,p}(\partial\Omega)} \|\nabla v\|_{L^{q}(\Omega)}.
\end{align*}
By a duality argument, it suffices to show that
\begin{equation}\label{6.7}
  \int_{\Omega}|\nabla v|^q\, dx\leq C\int_{\Omega}\frac{|F(x)|^q}{m(x,V)^q}\, dx .
\end{equation}
To show \eqref{6.7}, note that
\begin{equation}\label{6.11}
  \begin{aligned}
    |\nabla v(x)|&=\left|\int_{\Omega} \nabla_x N(x, y)F(y)\,d\sigma(y)\right|\\
    &\leq C\left(\int_{\Omega} |\nabla_x N(x, y)|m(y,V)^p\,dy\right)^{\frac1p} \left(\int_{\Omega} |\nabla_x N(x, y)|\cdot\frac{|F(y)|^q}{m(y,V)^q}\,dy\right)^{\frac1q}.
  \end{aligned}
\end{equation}
A similar computation as \eqref{6.8} shows
\begin{equation}\label{6.9}
\int_{\Omega} |\nabla_x N(x, y)|m(y,V)^p\,dy\leq Cm(x,V)^{p-1}.
\end{equation}
Plugging \eqref{6.8} and \eqref{6.9} into \eqref{6.11} gives that
  \begin{align*}
    \int_{\Omega}|\nabla v|^q\, dx&\leq C\int_{\Omega} \frac{|F(y)|^q}{m(y,V)^q} \int_{\Omega} m(x,V)|\nabla_x N(x, y)|\,dx dy\leq C\int_{\Omega} \frac{|F(y)|^q}{m(y,V)^q}dy.
  \end{align*}
In view of \eqref{6.10} and \eqref{Vu-estimate}, the proof is complete.
\end{proof}


\section{\bf Appendix: \texorpdfstring{$L^p$}{} Neumann problem}

In this section, we are going to establish $L^p$ estimate for Neumann problem. We remark here that our result in this section extends the result in \cite{Shen-1994-IUMJ} to $1<p<2+\e$, even in the case of $-\Delta u+V u=0$ in $\Omega$ and $\frac{\partial u}{\partial n}=g$ on $\partial\Omega$.

  \begin{thm}\label{main-thm-Lp-n}($L^p$ Neumann estimates)
Let $A$ satisfy \eqref{ellipticity}-\eqref{Dini} and $V>0$ a.e. satisfy \eqref{V}. There exists $\e>0$ such that for given $g\in L^{p}(\partial\Omega)$, $1< p<2+\e$, the weak solutions to
\begin{equation}\label{NP}
\mathcal{L} u+V u =0 \quad  \text{in } \Omega,\quad\frac{\partial u}{\partial \nu}  =g \quad \text{on } \partial\Omega\quad\text{and}\quad(\nabla u)^*\in L^p(\partial\Omega)
\end{equation}
satisfies the estimate
\begin{equation}\label{np-estimate}
\|(\nabla u)^*\|_{L^p(\partial\Omega)}\leq C\|g\|_{L^p(\partial\Omega)},
\end{equation}
where $C$ depends only on $\mu$, $d$, $p$ and the Lipschitz character of $\Omega$.
\end{thm}

\begin{proof}
The proof of the case $p=2$ follows by the well-known layer potential method. Let $f\in L^2(\partial\Omega)$, then $u(x)=\mathcal{S} (f) (x)$
is a solution to the $L^2$ Neumann problem in $\Omega$ with boundary data $(\frac{1}{2}I+K_A)f$. As a result, one may establish the existence of solutions of the $L^2$ Neumann problem in $\Omega$ by showing that the operators $\frac{1}{2}I+K_A: L^2(\partial\Omega)\to L^2(\partial\Omega)$ is invertible.

It suffices for us to show that \eqref{np-estimate} holds for $1< p<2$ and $2< p<2+\e$. We follow the approach in \cite{Dahlberg-Kenig-1987} to treat the case $1< p<2$. For solutions of the $L^p$ Neumann problem with $\mathcal{H}^1_{at}$ data, we have $\int_{\partial\Omega}(\nabla u)^*\,d\sigma\le C$. (See \cite{D-book,Grafakos-book,Stein-book} for a detailed presentation of $\mathcal{H}^1_{at}$)
This, together with the $L^2$ estimates, yields the desired ranges $1<p<2$, by the interpolation result, Theorem \ref{interpolation}.

Our approach to the range $2<p<2+\e$, which is based on a real variable argument.

To handle this, let $f\in L^2(\partial\Omega)$ with compact support and $T(f)=(\nabla u)^*$. To obtain the $L^p$($2<p<2+\e$) boundedness of $T$, apply Theorem \ref{real variable method} with $F=T(f)$, $F_B=T(f\chi_{6B})$ and $R_B=T(f\chi_{\partial\Omega\backslash 6B})$. It can be verified that $F\leq F_B+R_B$ and by $L^2$ solvability of the Neumann problem,
$$\int_{\Delta(x_0,r)}|F_B|^2\,d\sigma\leq C\int_{\Delta(x_0,6r)}|g|^2\,d\sigma.$$
Following the same line as the proof of \eqref{rh2}, we obtain that for some $q>2$,
\begin{align*}
  \left(\fint_{\Delta(x_0,r)} |R_B|^q d\sigma\right)^{1/q}&\leq C\left(\fint_{\Delta(x_0,2r)} |R_B|^2 d\sigma\right)^{1/2}\\
  &\leq C\left(\fint_{\Delta(x_0,2r)} |(\nabla u)^*|^2 d\sigma\right)^{1/2}+C\left(\fint_{\Delta(x_0,6r)} |g|^2 d\sigma\right)^{1/2}.
\end{align*}
Thus, by Theorem \ref{real variable method}, we obtain $\|(\nabla u)^*\|_{L^p(\partial\Omega)}\leq C\|g\|_{L^p(\partial\Omega)}$ for any $2<p<q$.
The proof completes.
\end{proof}

{\noindent\bf Data availability} Data sharing is not applicable to this paper as no data sets were generated or analysed.
\section*{Declarations}
{\bf Conflict of interest} The authors declared that they have no conflict of interest.


\newpage
\bibliographystyle{amsplain}
\bibliography{Manuscript-Geng-Xu}

\end{document}